\documentclass[11pt]{article}
 \usepackage[margin=1in]{geometry}
\usepackage{float}
\usepackage{amsmath,amsfonts,mathrsfs,wasysym,times,dsfont}
\usepackage{amsmath,amssymb,amsfonts,amsthm,amscd}
\usepackage{mathrsfs,latexsym,url,graphicx,enumerate}
\usepackage{booktabs}
\usepackage{array,multirow}
\usepackage{xcolor}
\usepackage{bbm}
\usepackage[utf8]{inputenc}
\usepackage{amsthm,amscd,amsfonts,newlfont}
\usepackage{amssymb, upref, color}
\usepackage{epsf, subfigure, verbatim}
\usepackage{latexsym}
\usepackage[colorlinks]{hyperref}
\usepackage{mathrsfs}
\usepackage{multirow}
\usepackage{lipsum,multicol}
\usepackage{setspace}
\usepackage{xcolor}

\usepackage{enumitem}

\DeclareMathOperator{\sign}{sign}

\numberwithin{equation}{section}

\theoremstyle{plain}
\newtheorem{exam}{Example}[section]
\newtheorem{theorem}[exam]{Theorem}
\newtheorem{lemma}[exam]{Lemma}
\newtheorem{remark}[exam]{Remark}

\newtheorem{proposition}[exam]{Proposition}
\newtheorem{definition}[exam]{Definition}
\newtheorem{corollary}[exam]{Corollary}

\title{Global Hopf bifurcation and connected components in a delayed predator-prey model}

\author{Wael El Khateeb\footnotemark[2] ,\ \  Guihong Fan\footnotemark[3] , \
\ Chunhua Shan\footnotemark[2] ,\  \ and Hao Wang\footnotemark[4]\ $^,$\footnotemark[1]}
\date{}

\begin{document}
\maketitle

\begin{abstract}
We study the dynamics of a delayed predator-prey system with Holling type II functional response, focusing on the interplay between time delay and carrying capacity. Using local and global Hopf bifurcation theory, we establish the existence of sequences of bifurcations as the delay parameter varies, and prove that the connected components of global Hopf branches are nested under suitable conditions. A novel contribution is to show that the classical limit cycle of the non-delayed system belongs to a connected component of the global Hopf bifurcation in Fuller’s space. Our analysis combines rigorous functional differential equation theory with continuation methods to characterize the structure and boundedness of bifurcation branches. We further demonstrate that delays can induce oscillatory coexistence at lower carrying capacities than in the corresponding ODE model, yielding counterintuitive biological insights. The results contribute to the broader theory of global bifurcations in delay differential equations while providing new perspectives on nonlinear population dynamics.

\vskip 0.2cm

\noindent {\bf Keywords.} Periodic solutions; Global Hopf bifurcation; Nested connected component; Combined effects of delay and carrying capacity; Fuller’s space.
\end{abstract}


\footnotetext[2]{Department of Mathematics and Statistics, The University of Toledo, Toledo, OH 43606, USA.}
\footnotetext[3]{Department of Mathematics, Columbus State University, Columbus, GA 31907, USA}
\footnotetext[4]{Department of Mathematical and Statistical Sciences, University of Alberta, Edmonton, Alberta, T6G 2G1, Canada.}
\footnotetext[1]{Corresponding author. {\it E-mail address:} hao8@ualberta.ca (H. Wang).}

\section{Introduction}

In the early twentieth century, Lotka and Volterra pioneered the mathematical study of population dynamics, introducing models that remain central to theoretical ecology. Their work provided a foundational framework for describing predator-prey interactions and inspired a wide range of generalizations. A commonly studied form of the generalized Lotka-Volterra model is
\begin{equation}
    \begin{cases}
         \dfrac{dx(t)}{dt}= x\alpha(x)- p(x)y, \\
         \dfrac{dy(t)}{dt}=cp(x)y-dy,
    \end{cases}
\end{equation}
where $x(t)$ and $y(t)$ denote the prey and predator densities at time $t$, respectively. Here, $\alpha(x)$ describes the per capita prey growth rate in the absence of predation, $c$ is the conversion efficiency of prey into predator biomass, $d$ is the predator mortality rate, and $p(x)$ is the functional response characterizing the predation rate as a function of prey density.

To capture the biological reality that consumption does not instantly translate into predator reproduction, a time delay $\tau$ is often introduced. Specifically, $\tau$ represents the period between prey capture and its conversion into viable predator biomass. The survival probability of predators during this interval is given by $e^{-d\tau}$ \cite{Wolkowicz2006}. Assuming logistic growth for the prey population and a Holling type II functional response, we obtain the delayed predator-prey system
\begin{equation} \label{modelH2}
    \begin{cases}
        \dfrac{dx(t)}{dt}=rx(t)\!\left(1-\dfrac{x(t)}{K}\right)-\dfrac{mx(t)y(t)}{a+x(t)}, \\[1.2ex]
        \dfrac{dy(t)}{dt}=-dy(t)+e^{-d\tau}\dfrac{cmx(t-\tau)y(t-\tau)}{a+x(t-\tau)},
    \end{cases}
\end{equation}
where $r$ is the intrinsic prey growth rate, $K$ is the environmental carrying capacity, $a$ is the half-saturation constant, and $m$ is the predation rate. All parameters are positive. The incorporation of $\tau$ renders the model more realistic by reflecting the maturation time of predators and compensating for mortality during the delay.

Delayed predator-prey systems of this type have been extensively studied from both biological and mathematical perspectives; see \cite{Cook2006, Fan2021, Kuang2004, Wang2014, Li2023, Mai2019, Song2005, Wang2009, Ruan2001, Xu2012, Shu2024} and references therein. An alternative biological interpretation of the delay and survival term $e^{-d\tau}$ was proposed by Gourley and Kuang \cite{Kuang2004}, who studied stage-structured models in which juvenile predators experience a mortality rate $d_j \neq d$ during the maturation period $\tau$, leading to the survival factor $e^{-d_j\tau}$. Related stage-structured models with Holling type I and II responses were analyzed by Gourley and Kuang \cite{Kuang2004} and by Li, Lin, and Wang \cite{Wang2014} using Hopf bifurcation analysis. Cooke, Elderkin, and Huang \cite{Cook2006} incorporated juvenile prey mortality into similar delayed systems, while Fan and Wolkowicz \cite{Fan2021} considered model \eqref{modelH2} with a Holling type I response, demonstrating the existence of local Hopf bifurcations and a cascade of period-doublings leading to chaos. Diffusive extensions with predator maturation delay have also been investigated, for instance by Xu et al.\ \cite{Shu2024} in the context of global Hopf bifurcation. More broadly, delay has been employed in predator-prey models to capture diverse biological mechanisms and to explore their rich dynamical behavior \cite{Cook1999, Shu2015, Ruan2001, Li2023, Song2005, Wang2009, Xu2012, Mai2019}.

It is well known that the corresponding ordinary differential equation (ODE) model obtained by setting $\tau=0$ in \eqref{modelH2} admits a unique globally asymptotically stable periodic orbit. The focus of this paper is to uncover the richer dynamics that emerge for $\tau > 0$, with particular emphasis on the global structure of Hopf bifurcation. Our main results are as follows:
\begin{itemize}
\setlength\itemsep{0.25em}
\item There exists a sequence of local Hopf bifurcations as $\tau$ (or $K$) varies. Under suitable conditions, the connected components of the global Hopf bifurcation are nested with respect to~$\tau$.
\item The limit cycle of the ODE model belongs to a connected component of the global Hopf bifurcation of the DDE model \eqref{modelH2} in Fuller’s space. This connection between ODE periodic orbits and DDE Hopf components appears to be new.
\item Introducing a positive delay $\tau$ can lead to predator-prey coexistence through oscillatory dynamics at lower carrying capacities compared to the delay-free case, a counterintuitive phenomenon given the mortality losses incurred during the delay.
\end{itemize}

The novelty of the second result lies in establishing a precise relationship between ODE limit cycles and global Hopf components in delay differential equations, which to the best of our knowledge has not been reported in the context of Holling-type functional responses. The third result offers new biological insight by demonstrating that delay can in fact facilitate coexistence under more restrictive environmental conditions.

The paper is organized as follows. In Section~2, we establish the well-posedness of \eqref{modelH2} with prescribed initial data and recall the known dynamics of the corresponding ODE model. Section~3 analyzes local dynamics near equilibria and the onset of Hopf bifurcations. Section~4 develops the theory of global Hopf bifurcation and connected components. In Section~5, we present bifurcation diagrams in the $(\tau,K)$-plane, highlighting the joint effects of delay and carrying capacity. Section~6 concludes with biological interpretation and significance of the results.

\section{Preliminaries}

\subsection{Existence, uniqueness, and boundedness of solutions}

Let $\mathbf{C}=C([-\tau,0],\mathbb{R}_+)$ be a set of continuous functions mapping $[-\tau,0]$ to $\mathbb{R}_+$ equipped with the supremum norm, where $\mathbb{R}_+$ is the set of positive real numbers. Consider the initial data $\phi=(\phi_1(t),\phi_2(t))\in\mathbf{C\times C}$  for system \eqref{modelH2}.
The right-hand side of system \eqref{modelH2} is continuously differentiable with respect to the arguments $x(t)$, $y(t)$, $x(t-\tau)$, and $y(t-\tau)$.  By the existence and uniqueness theorem \cite{Hale1993, Kuang1993}, there exists a unique solution $(x_1(t), x_2(t))$ for $t$ in a small neighborhood of $t=0$, and satisfies the initial condition $(x_1(t), x_2(t))=(\phi_1(t),\phi_2(t))$ for $t\in [-\tau, 0]$. Moreover, the solution is continuously differentiable.

\begin{lemma}\label{positivity}
$x(t)$ and $y(t)$ are positive for $t>0$ given the initial condition $\phi\in\mathbf{C\times C}$.
\end{lemma}
\begin{proof} It follows from the first equation of system \eqref{modelH2} and the formula of variation of constants that
$$x(t)=x(0)e^{ \int_{0}^{t} r(1-\frac{x(s)}{K})-\frac{my(s)}{a+x(s)}\,ds }.$$
Since $x(0)>0$, then $x(t)>0$ for $t>0$ on the interval of existence.
For $y(t)$, since $y(0)>0$ and $y(t)$ is continuous, then either $y(t)>0$ for $t>0$ on the interval of existence, or  $\exists$ $t_0>0,$ such that $y(t_0)=0$, but $y'(t_0)\leq 0$. However,  by substituting $t=t_0$ into the second equation of system \eqref{modelH2}, we obtain
$$
y'(t_0)=e^{-d\tau}\frac{cmx(t_0-\tau)y(t_0-\tau)}{a+x(t_0-\tau)}>0,
$$
Which is a contradiction. Hence, $y(t)$ is positive for $t>0$.
\end{proof}

\begin{lemma}\label{boundedness}
       The solution $(x(t), y(t))$ is eventually uniformly bounded from above provided that the initial condition $\phi\in\mathbf{C\times C}$.
\end{lemma}

\begin{proof}
    We will first show that $x(t)$ is eventually uniformly bounded from above. From the first equation of system \eqref{modelH2}, by Lemma \ref{positivity},  we have
$$\frac{dx}{dt}=rx(t)(1-\frac{x(t)}{K})-\frac{mx(t)y(t)}{a+x(t)}<rx(t)(1-\frac{x(t)}{K}).$$
Let $X(t)$ be the solution to $\frac{dX}{dt}=X(t)(1-\frac{X(t)}{K})$ with $X(0)=x(0)>0$. By the Comparison Theorem,  $$\limsup_{t\rightarrow \infty}x(t)\leq \limsup_{t\rightarrow \infty}X(t)=K.$$ To show $y(t)$ is eventually uniformly bounded from above, let $U(t)=ce^{-d\tau}x(t-\tau)+y(t)$, then
\begin{eqnarray*}
    \frac{dU}{dt}&=&ce^{-d\tau}\frac{d(x(t-\tau))}{dt}+\frac{d(y(t))}{dt}\\
                 &=&ce^{-d\tau}\left((r+d)x(t-\tau)-\frac{r(x(t-\tau))^2}{K}\right)-dU(t).
\end{eqnarray*}
Notice that $$(r+d)x(t-\tau) -\frac{r(x(t-\tau))^2}{K}-\frac{(r+d)^2K}{4r}=-\left(\frac{\sqrt{r}x(t-\tau)}{\sqrt{K}}-\frac{(r+d)\sqrt{K}}{2\sqrt{r}}\right)^2\leq 0,$$
and accordingly,
$$\frac{dU}{dt}\leq ce^{-d\tau}\frac{(r+d)^2K}{4r}-dU(t).$$
Let $Y(t)$ be the solution to  $\frac{dY}{dt}=ce^{-d\tau}\frac{(r+d)^2K}{4r}-dY(t)$ with $Y(0)=U(0)$.
Then
$$Y(t)=\frac{ce^{-d\tau}(r+d)^2K}{4rd}+\left(U(0)-\frac{ce^{-d\tau}(r+d)^2K}{4rd}\right)e^{-dt}.$$
By the Comparison Theorem, $$\limsup_{t\rightarrow \infty}y(t)\leq \limsup_{t\rightarrow \infty}U(t)\leq \limsup_{t\rightarrow \infty}Y(t)=\frac{ce^{-d\tau}(r+d)^2K}{4rd}\leq cK.$$
Hence, $y(t)$ is eventually uniformly bounded from above.
\end{proof}

\begin{theorem}
  The solution $(x(t), y(t))$ of system \eqref{modelH2} exists on $[-\tau, \infty)$ for $\phi\in\mathbf{C\times C}$.
\end{theorem}
\begin{proof}
By Lemma \ref{positivity} and Lemma \ref{boundedness}, the solution $(x(t), y(t))$ is eventually uniformly bounded. By the Continuation Theorem \cite{Hale1993, Kuang1993}, the solution exists for all $t>0$.
\end{proof}


\subsection{Dynamics of the classical model and equilibrium points of system \texorpdfstring{\eqref{modelH2}}{}}

To fully understand the dynamics of the delayed model \eqref{modelH2}, we recall the dynamics of the classical model below, i.e., the corresponding ODE system of model \eqref{modelH2}, which was studied by Jing and Chen in the 1980s \cite{Jing1984}.

\begin{equation}\label{ODE model}
    \begin{cases}
    \displaystyle \frac{dx(t)}{dt}=rx(1-\frac{x}{K})-\frac{mxy}{a+x},\\
    \displaystyle \frac{dy(t)}{dt}=y(-d+\frac{cmx}{a+x}).
\end{cases}
\end{equation}

\begin{lemma}\label{ODEsystemdynamics}
Let $K_c=\frac{ad}{cm-d}$, and $K_0=\frac{a(cm+d)}{cm-d}$.
System \eqref{ODE model} has three equilibrium points $E_0(0,0)$,  $E_K(K,0)$ and  $E^*(\frac{ad}{cm-d}, \frac{rca(Kcm-Kd-ad)}{K(cm-d)^2})$, where $E_0$ is always a saddle.
\begin{itemize}

 \item $E_K$ is globally asymptotically stable for $K<K_c$ and becomes unstable for $K>K_c$.

 As $K$ increases and crosses the threshold $K=K_c$, a transcritical bifurcation takes place between $E_K$ and $E^\ast$. $E_K$ loses stability and $E^\ast$ enters the first quadrant.

  \item   $K=K_0$ is the Hopf bifurcation threshold for $E^*$. $E^\ast$ is globally asymptotically stable for $K_c<K<K_0$ and unstable for $K>K_0$.

  \item  System \eqref{ODE model} has no periodic solution for $K\leq K_0$, and a unique periodic solution for $K>K_0$.

\end{itemize}
\end{lemma}

For system (\ref{modelH2}), it is easy to see that the equilibrium points $E_0(0, 0)$ and $E_K(K, 0)$ exist for all $\tau>0$. To explore the existence and stability of interior equilibrium points, we define the following parameters for convenience: $$K_1=\frac{ad}{cme^{-d\tau}-d}, K_2=\frac{ad(3cm+d)}{c^2m^2-d^2}, \tau_{max}=\frac{1}{d} \ln\left(\frac{Kcm}{ad+dK}\right).$$ Then $K_c<K_2<K_0$ and $K_c<K_1$. We introduce the condition below.

\begin{itemize}
    \item Condition $C_0: cm>d$,  and $K>K_1$ (i.e., $\tau< \tau_{max}$).



    \end{itemize}

Through straightforward calculations, we derive the following lemma.

\begin{lemma}\label{lemmaE}
The interior equilibrium $E^*(\frac{da}{ce^{-d\tau}m-d},\frac{rce^{-d\tau}a(Kce^{-d\tau}m-Kd-da)}{K(ce^{-d\tau}m-d)^2})$ of system \eqref{modelH2} exists under the condition $C_0$.
\end{lemma}

Consider the stability of equilibrium points. Let $(x_0, y_0)$ be any equilibrium point of system \eqref{modelH2} The characteristic matrix at $(x_0, y_0)$ is given by
$$ J(x_0,y_0)=\begin{bmatrix}
\bigg(r-\frac{2rx_0}{K}-\frac{may_0}{(a+x_0)^2}\bigg) -\lambda& \frac{-mx_0}{a+x_0} \\
e^{-\lambda \tau}\frac{ce^{-d\tau}may_0}{(a+x_0)^2} & \bigg(-d+e^{-\lambda \tau}\frac{ce^{-d\tau}mx_0}{(a+x_0))}\bigg)-\lambda
\end{bmatrix}.$$

By substituting $x_0$ and $y_0$ with 0 in $J(x_0,y_0)$, we obtain
two simple eigenvalues of $J(0, 0)$:  $\lambda_1=r>0$ and $\lambda_2=-d<0$. Applying the Hartman-Grobman Theorem \cite{Hartman}, we establish the following lemma.
\begin{lemma}\label{origin}
Equilibrium $E_0$ is a saddle point for all $\tau\geq 0$.
\end{lemma}




\section{Local dynamics near equilibrium points}

\subsection{Local stability of \texorpdfstring{$E^*$}{}}

Substituting the coordinates of $E^*$ into $J(x_0, y_0)$, we obtain
$$
 J(E^\ast)
 =
\begin{bmatrix}
\frac{rd(K(ce^{-d\tau}m-d)-a(ce^{-d\tau}m+d))}{Kce^{-d\tau}m(ce^{-d\tau}m-d)} -\lambda & \frac{-d}{ce^{-d\tau}}\\
e^{-\lambda\tau}\frac{r(Kce^{-d\tau}m-Kd-da)}{Km} & -d+de^{-\lambda \tau}-\lambda
\end{bmatrix}.$$\\
The characteristic equation is derived by setting the determinant of $J(E^\ast)$ to zero, i.e.,
\begin{equation}\label{characteristics equation E^* Holling type 2}
    P(\lambda,\tau):=\lambda^2+[d-H(\tau)]\lambda+e^{-\lambda\tau}[-d\lambda+L(\tau)]-dH(\tau)=0,
\end{equation}
where $H(\tau)$ and $L(\tau)$ are defined as follows
$$H(\tau)=r\bigg[\frac{de^{d\tau}}{cm}-\frac{ad}{K(ce^{-d\tau}m-d)}-\frac{ad^2e^{d\tau}}{(ce^{-d\tau}m-d)Kcm}\bigg],$$ $$L(\tau)=dH(\tau)+A(\tau), \ \text{and}\ \  A(\tau)=rd\left(\frac{Kce^{-d\tau}m-Kd-da}{Kmce^{-d\tau}}\right),$$
with $\tau\in[0, \tau_{max})$.  Under condition $C_0$, $A(\tau)$ is strictly positive, and the functions $H(\tau)$, $L(\tau)$, and $A(\tau)$ are well defined and bounded. For simplicity, we will drop the $(\tau)$ and express $H(\tau)$, $L(\tau)$ and $A(\tau)$ as $H, L$, and $A$ when necessary.

\begin{lemma}\label{roots}
The sum of the multiplicity of roots of $P(\lambda, \tau)=0$  on $\mathbb{C}^+=\{\lambda\in \mathbb{C}: Re(\lambda)>0\}$ can change only if 
a pair of conjugate complex roots appear on or cross the imaginary axis as $\tau$ varies continuously on $[0, \tau_{max})$.

\end{lemma}

\begin{proof}
To show that no eigenvalue emerges from infinity, we write $P(\lambda,\tau)= \lambda^2 +\eta(\lambda,\tau)$, where $\eta(\lambda,\tau)=[d-H(\tau)]\lambda+e^{-\lambda\tau}[-d\lambda+L(\tau)]-dH(\tau)$. Since $H(\tau)$ and $L(\tau)$ are bounded on $[0, \tau_{max})$, we have
$$\limsup_{Re\lambda>0, |\lambda| \to \infty }\big|\lambda^{-2} \eta(\lambda,\tau)\big|\leq \limsup_{Re\lambda>0, |\lambda| \to \infty }\left(\frac{2d+|H(\tau)|}{|\lambda|}+\frac{|L(\tau)|+d|H(\tau)|}{|\lambda|^2}\right)=0<1.$$
Furthermore, $P(0, \tau)=A(\tau)>0$ for $\tau\in[0, \tau_{max})$. Hence, $\lambda=0$ is not a root for $\tau\in[0, \tau_{max})$. By Theorem 1.4, Chapter 3 of \cite{Kuang1993}, the desired result is obtained.
\end{proof}

To study the stability of $E^\ast$ and the possibility of Hopf bifurcation at $E^*$ as $\tau$ varies, by Lemma \ref{roots} we assume $\lambda=\pm iw$ is a pair of purely imaginary roots of $P(\lambda, \tau)=0$ where $w>0$. This leads to
$$-w^2+ [d-H]iw+e^{-i\tau w}[-diw+L]-dH=0.$$
Using Euler's formula and equating the real and imaginary parts to zero, we obtain
\begin{equation}
\begin{cases}
    -w^2 +L\cos(\tau w)-dw\sin(\tau w)-dH=0,\\
    (d-H)w -dw\cos(\tau w)-L\sin(\tau w)=0.
\end{cases}
\end{equation}
Solving for $\cos(\tau w)$ and $\sin(\tau w)$, and using the assumption $\tau \in [0,\tau_{max})$ we obtain
\begin{equation}\label{sin and cos}
    \begin{cases}
        \sin(\tau w)=\frac{dwL-HwL-d^2wH-dw^3}{L^2+d^2w^2},\\
        \cos(\tau w)=\frac{LdH+Lw^2+(dw)^2-dHw^2}{L^2+d^2w^2}.
    \end{cases}
\end{equation}
Both $\sin(\tau w)$ and $\cos(\tau w)$ are well defined, since $L^2\geq 0$ and $d^2w^2>0$. Applying $\sin^2(\tau w)+\cos^2(\tau w)=1$ to \eqref{sin and cos}, we obtain
\begin{equation}\label{w polynomial}
(d^2w^2+L^2)(w^4+H^2w^2+d^2H^2-L^2)=0.
\end{equation}

\begin{theorem}\label{wtau theorem}
Let $K>K_2$, and define
$$\bar{\tau}=
\frac{1}{d}\ln\left(\frac{cm\sqrt{9a^2+4aK+4K^2}-3acm}{2\left(ad+dK\right)}\right).
$$
The following statements hold.
\begin{enumerate}

\item[(1)] $0<\bar{\tau}<\tau_{max}$.

\item[(2)]  $L(\tau)+dH(\tau)>0$ if $\tau\in[0, \bar{\tau})$, $L(\tau)+dH(\tau)=0$ if $ \tau=\bar{\tau}$, and $L(\tau)+dH(\tau)<0$ if $\tau\in(\bar{\tau}, \tau_{max})$.

\item [(3)] The unique solution of the hexic equation (\ref{w polynomial}) for $0\leq\tau< \bar{\tau}$ and $K>K_2$ is given by \begin{equation}  \label{wtauexpression}
w(\tau)=\sqrt{\frac{\sqrt{((H(\tau))^4 - 4d^2(H(\tau))^2 + 4(L(\tau))^2)}}{2} - \frac{(H(\tau))^2}{2}}.\end{equation}
\end{enumerate}
\end{theorem}

\begin{proof}
Statement (1) is obtained by direct calculation.

(2). $L(\tau)+dH(\tau)$ is continuous for $\tau\in[0, \tau_{max})$. It is easy to check that $L(\tau)+dH(\tau)$ equals zero if and only if $\tau=\bar{\tau}$, and $$L'(\bar{\tau})+dH'(\bar{\tau})=-\frac{4ard^2(K+a)(\sqrt{4K^2+4Ka+9a^2}-3a)}{(2K+5a-\sqrt{4K^2+4Ka+9a^2})^2K}<0.$$  
Hence, the sign of $L(\tau)+dH(\tau)$ is obtained for $\tau\in[0, \tau_{max})$.

(3) Let $Z=w^2$, and substitute $Z$ into the hexic equation \eqref{w polynomial}. Then
three solutions appear
$$
Z_{\pm}=\frac{\pm\sqrt{H^4 - 4d^2H^2 + 4L^2}}{2}-\frac{H^2}{2}, Z_3=\frac{-L^2}{d^2}.$$
Since $Z=w^2>0$, $Z_{-}$ and $Z_3$ are discarded. To ensure $w=\sqrt{Z_+}$ is well defined, it suffices to show  $$(H(\tau))^4 - 4d^2(H(\tau))^2 + 4(L(\tau))^2> (H(\tau))^4,$$
i.e., $(L(\tau))^2>d^2(H(\tau))^2$. Since $A(\tau)$ is positive for $\tau\in [0, \tau_{max})$, we have $L(\tau)=dH(\tau)+A(\tau)>dH(\tau)$.  Moreover, by statement (2), we have $L(\tau)>-dH(\tau)$ for $\tau\in [0, \bar{\tau})$. In conclusion, $L(\tau)>|dH(\tau)|$, and $(L(\tau))^2>d^2(H(\tau))^2$.
\end{proof}

\begin{corollary}
    The equilibrium $E^\ast$ is locally asymptotically stable if $K\leq K_2$ and $\tau\in(0, \tau_{max})$ , or $K_2<K<K_0$ and $\tau\in[\bar{\tau}, \tau_{max})$.
\end{corollary}
\begin{proof}
Based on Theorem \ref{wtau theorem}, $d^2(H(\tau))^2-(L(\tau))^2>0$ if $K\leq K_2$ and $\tau\in(0, \tau_{max})$ , or $K_2<K<K_0$ and $\tau\in[\bar{\tau}, \tau_{max})$. Then \eqref{w polynomial} has no positive root and the desired result follows.
\end{proof}

\subsection{Local Hopf bifurcation at \texorpdfstring{$E^*$}{}}

By statement (3) of Theorem \ref{wtau theorem}, we assume  $0\leq\tau\leq \bar{\tau}<\tau_{max}$ and $K>K_2$, which
guarantees that the hexic equation \eqref{w polynomial} has a positive root $w=w(\tau)$, a necessary condition for Hopf bifurcation.

We explore the critical values for $\tau\in [0, \bar{\tau})$ at which Hopf bifurcation occurs. From the second equation of system \eqref{sin and cos}, it is reasonable to define the angle $\tau w(\tau)$ using the inverse cosine function, while also considering the sign of $\sin(\tau w(\tau))$. Here $\sin(\tau w(\tau))$ is the composition of the functions $\sin(\tau w)$ and $w(\tau)$, which are given by the first equation of system \eqref{sin and cos} and equation \eqref{wtauexpression}, respectively.


\begin{lemma}\label{sign of theta}
  Define $\Breve{\tau}={\frac {1}{d}\ln  \left( {\frac {cm \left( K-a \right) }{d \left(K+a
 \right) }} \right) }
$. 
\begin{enumerate}

\item[(1)] If $K< K_0$, then $\Breve{\tau}<0$, or $\breve{\tau}$ does not exist,  and  $\sin(\tau w(\tau))> 0$  for all $\tau\in[0,\bar{\tau})$. 

\item[(2)] If $K=K_0$, then $\breve{\tau}=0$, $\sin(\breve{\tau} w(\breve{\tau}))= 0$. Moreover, $\sin(\tau w(\tau))> 0$  for all $\tau\in(0,\bar{\tau})$.

\item[(3)] If $K>K_0$, $\breve{\tau}$ is well defined with $0<\breve{\tau}<\bar{\tau}$.  Then $\sin(\tau w(\tau))$ changes sign as follows:

$\sin(\tau w(\tau))<0$  for $0<\tau<\breve{\tau}$; $\sin(\tau w(\tau))=0$  for $\tau=\breve{\tau}$;  $\sin(\tau w(\tau))>0$ for $\breve{\tau}<\tau<\bar{\tau}$.
\end{enumerate}
Consequently, the function $\theta_n(\tau), \tau\in[0, \bar{\tau})$ is well defined as follows.

\begin{itemize}
    \item If $\sin(\tau w(\tau))<0$:  $\theta_n(\tau)=-\arccos\left(\frac{LdH+Lw^2+d^2w^2-dHw^2}{L^2+d^2w^2}\right)+2n\pi.$

    \item If $\sin(\tau w(\tau))\geq 0$: $\theta_n(\tau)=\arccos\left(\frac{LdH+Lw^2+d^2w^2-dHw^2}{L^2+d^2w^2}\right)+2n\pi.$

\end{itemize}

Here $n\in \mathbb{N}$, and $\mathbb{N}$ is the set of all nonnegative integers.
\end{lemma}

\begin{proof}
Solving $\sin(\tau w(\tau))=0$ for $H(\tau)$, we obtain three solutions $$H(\tau)=0, H(\tau)=\frac{L(\tau)}{d}, H(\tau)=2d.$$
The solution $H(\tau)=\frac{L(\tau)}{d}$ is discarded since it is valid only for $\tau=\tau_{max}$. The third solution $H(\tau)=2d$ is also discarded. Indeed, substituting $H(\tau)=2d$ into $\sin(\tau w(\tau))=0$ yields $|L(\tau)|+L(\tau)=0$. Since $A(\tau)>0$, we have
$$0<2d^2+A(\tau)=dH(\tau)+A(\tau)=L(\tau)=-|L(\tau)|\leq 0,$$ which leads to a contradiction. Therefore, the remaining solution $H(\tau)=0$ holds when $\tau=\breve{\tau}$.

Case (1). If $K<K_0$, then $\Breve{\tau}< 0$, or $\breve{\tau}$ does not exist. Then $H(\tau)<0$ for all $\tau\in[0, \bar{\tau})$. By continuity,  $\sin(\tau w(\tau))$ does not change sign as $\tau$ varies in $[0,\bar{\tau})$. We are now investigating the sign of $\sin(\tau w(\tau))$.

Since $L(\tau)=dH(\tau)+A(\tau)$ and $w(\tau)$ is given by \eqref{wtauexpression}, we obtain
\begin{equation*} \label{Inequality for sin theta}
        \sign\{\sin(\tau w(\tau)\}=
    \sign\{dL(\tau)-H(\tau)L(\tau)-d^2H(\tau)-d(w(\tau))^2\}=
   \sign\{\xi_1-\xi_2\},
    \end{equation*}
where $\xi_1=2dA(\tau)-d(H(\tau))^2-2A(\tau)H(\tau)$ and $
\xi_2=d\sqrt{(H(\tau))^4+8dA(\tau)H(\tau)+4(A(\tau))^2}.$

Since $ L(\tau)>-dH(\tau)$ for $\tau\in[0, \bar{\tau})$, it follows that $A(\tau)=L(\tau)-dH(\tau)>-2dH(\tau)$. Noting that $H(\tau)<0$, we get
$\xi_1> 3d(H(\tau))^2-4d^2H(\tau)=dH(\tau)(3H(\tau)-4d)>0$. It is apparent that $\xi_2\geq 0$.

Next, we analyze the sign of $\xi_1^2-\xi_2^2$. Since $H(\tau)<0$, we obtain
\begin{eqnarray*}
    \sign\{\xi_1^2-\xi_2^2\}&=&\sign\{ A(\tau) H(\tau) \left(H(\tau) -2 d \right) \left(dH(\tau)+A(\tau)+d^{2}\right)\}\\ &=&\sign\{dH(\tau)+A(\tau)+d^{2}\}>0.
\end{eqnarray*}
Since $A(\tau)>-2dH(\tau)$, it follows that $dH(\tau)+A(\tau)+d^{2}>-dH(\tau)+d^2>0$. Thus, $\xi_1^2-\xi_2^2>0,$ implying $\xi_1>\xi_2>0$.  Hence,
$$\sign\{\sin(\tau w(\tau)\}=\sign\{\xi_1-\xi_2\}=
\frac{\sign\{\xi^2_1-\xi^2_2\}}{\sign\{\xi_1+\xi_2\}} =1, \ \forall   \tau\in[0,\bar{\tau}).
$$

Cases (2) and (3). Since $K_0>a$, if $K\geq K_0$, then $\breve{\tau}$ is well defined.
In this case, to study the sign of $\sin(\tau w(\tau))$,  it suffices to check the sign of its derivative at $\breve{\tau}$. We have $$\frac{d}{d\tau}(\sin(\tau w(\tau)))\big|_{\tau=\breve{\tau}}=\frac{\sqrt{\frac{a r d}{K}}\, \left(a -K \right)^{2}}{2 a^{2}}>0,$$ and which implies that $\sin(\tau w(\tau))$ changes sign from negative to positive at $\Breve{\tau}$. The additional condition $K>K_0$ ensures that $\breve{\tau}$ is positive, and the desired results follow.
\end{proof}


\begin{figure}[H]
    \centering
   \subfigure[$K<K_0$]{\includegraphics[angle=0, width=0.43\textwidth]{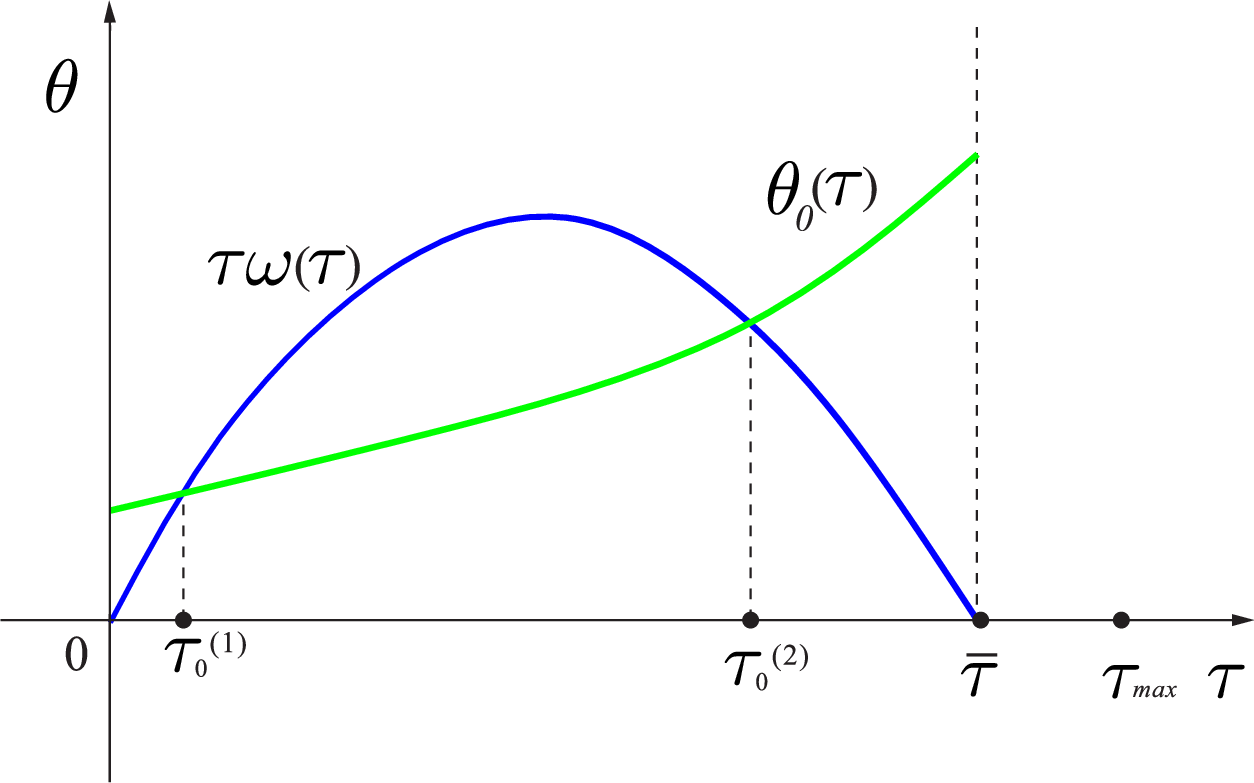}} \
   \subfigure[$K>K_0$]{\includegraphics[angle=0, width=0.43\textwidth]{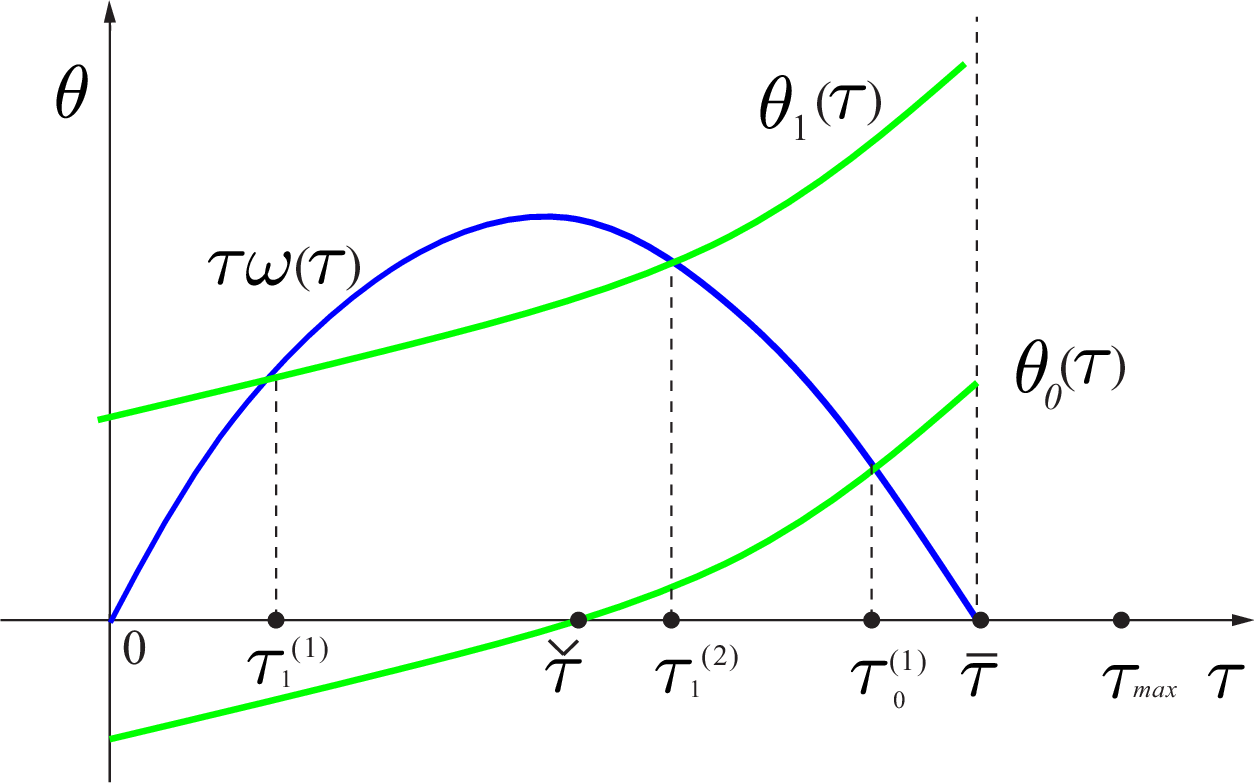}}
    \label{tau34}
    \caption{Intersection of $\tau w(\tau)$ and $\theta_n(\tau)$. Here, the graphs of $\tau w(\tau)$ and $\theta_n(\tau)$ are sketched in a relatively simple way, but ensuring that the properties of Lemma \ref{theta12} are all satisfied.}
\end{figure}


By Theorem \ref{wtau theorem} and Lemma \ref{sign of theta}, we obtain two functions $$\theta=\tau w(\tau), \ \theta=\theta_n(\tau), \tau\in [0, \bar{\tau}].$$  These two functions are defined at $\tau=\bar{\tau}$ by continuous extension. The following lemma is obvious.

\begin{lemma}\label{theta12}
Consider two functions $\theta=\tau w(\tau)$ and $\theta=\theta_n(\tau)$, where $n\in \mathbb{N}$.

(1) $\theta=\tau w(\tau)\geq 0$ on  $\tau\in[0, \bar{\tau}]$ and vanishes only at $\tau=0$ and $\tau=\bar{\tau}$.

(2) 
$\theta_0(0)>0 (=0, <0, resp.)$ if $K<K_0
 (=K_0, >K_0, resp)$.

 (3) $\theta_0(\bar{\tau})=\pi$ regardless of $K$.

 (4) If $K<K_0$, then there exists $\zeta>0$ such that  $0<\zeta\leq \theta_0(\tau)\leq \pi$ for $\tau \in [0,\bar{\tau}]$.
\end{lemma}

\begin{proof}
Statements (1) and (2) follow from Theorem \ref{wtau theorem} and Lemma \ref{sign of theta}.

For statements (3) and (4), a straightforward calculation gives $\sin(\bar{\tau} w(\bar{\tau}))=0$ and $\cos(\bar{\tau} w(\bar{\tau}))=-1$, which implies that $\theta_0(\bar{\tau})=\pi$. Since $\theta_0(\tau)$ is continuous on $[0, \bar{\tau}]$, it attains the maximum and minimum values on $[0, \bar{\tau}]$.   By statement (1) of Lemma \ref{sign of theta}, $\sin(\tau w(\tau))> 0$ for all $\tau\in[0,\bar{\tau})$, so $\zeta=\displaystyle\min_{\tau \in [0,\bar{\tau}]}\theta_0(\tau)>0$  and $\displaystyle\max_{\tau \in [0,\bar{\tau}]} \theta_0(\tau)=\pi$.
\end{proof}

 \begin{theorem}\label{lhopf}
     Assume that there exists $\tau^\star \in (0, \bar{\tau})$ such that $$\tau w(\tau)_{|\tau^\star}=\theta_n(\tau)_{|\tau^\star}, \ \text{and}\ \   (\tau w(\tau))'_{|\tau^\star}\neq (\theta_n(\tau))'_{|\tau^\star}$$ for some $n\in\mathbb{N}$. Then system \eqref{modelH2} undergoes a Hopf bifurcation at $\tau=\tau^\star$, and a family of non-constant
periodic solutions exists in a neighborhood of $E^\ast$ for $\tau$ in a small neighborhood of $\tau=\tau^\star$.
 \end{theorem}

 \begin{proof}
Let $\lambda(\tau)=\alpha(\tau)\pm i\beta(\tau)$ be a pair of conjugate complex eigenvalues of $P(\lambda, \tau)=0$. The condition $\tau w(\tau)_{|\tau^\star}=\theta_n(\tau)_{|\tau^\star}$ implies that $\alpha(\tau^\star)=0$ and $\beta(\tau^\star)=w(\tau^\star)>0$, meaning that $P(\lambda, \tau)=0$ has a pair of purely imaginary eigenvalues $\pm iw(\tau^\star)$ at $\tau=\tau^\star$. Next we check the transversality condition by defining the function $$S_n(\tau)=\tau-\frac{\theta_n(\tau)}{w(\tau)}, n\in\mathbb{N}.$$

By Theorem 2.2 of \cite{Kuang2002} (the geometric stability switch criterion), we obtain
\begin{equation*}\label{Beretta Thm}
\delta(\tau^\star):=\sign \bigg\{\frac{d \alpha(\tau)}{d\tau}\bigg |_{\tau=\tau^\star} \bigg\}=\sign \bigg\{\frac{d F(w(\tau))}{dw}\bigg |_{\tau=\tau^\star} \bigg\}\cdot \sign\bigg\{\frac{d S_n(\tau)}{d\tau}\bigg|_{\tau=\tau^\star} \bigg\},
\end{equation*}
where
$F(w(t))=(w(\tau))^4+(H(\tau))^2(w(\tau))^2+d^2(H(\tau))^2-(L(\tau))^2$. Thus,
$$F_w(w(\tau))=4(w(\tau))^3+2(H(\tau))^2w(\tau)>0, \, \forall \tau\in (0,\bar{\tau}).$$


By the definition of $S_n(\tau)$, we have
$$ \frac{dS_n(\tau)}{d\tau}=1-\frac{\theta'_n(\tau)w(\tau)-w'(\tau)\theta_n(\tau)}{(w(\tau))^2}=\frac{\tau(w(\tau))^2-[\theta_n'(\tau)\tau w(\tau)- \tau w'(\tau)\theta_n(\tau)]}{\tau(w(\tau))^2}.$$
Since $(\tau w(\tau))'=\tau w'(\tau) +w(\tau)$, we get $\tau  w'(\tau)=(\tau w(\tau))'-w(\tau)$. Noting that $ \theta_n(\tau^\star)=\tau^\star w(\tau^\star)$, we obtain
\begin{eqnarray*}
\frac{dS_n(\tau)}{d\tau}\bigg|_{\tau=\tau^\star}&=&\dfrac{\tau^\star(w(\tau^\star))^2-\theta_n'(\tau^\star)\tau^\star w(\tau^\star)+[(\tau w(\tau))'_{|\tau^\star}-w(\tau^\star)]\tau^\star w(\tau^\star)}{\tau^\star(w(\tau^\star))^2}\\
&=&\dfrac{\tau^\star w(\tau^\star)[(\tau w(\tau))'_{|\tau^\star}-\theta_n'(\tau^\star)]}{\tau^\star(w(\tau^\star))^2}=\dfrac{(\tau w(\tau))'_{|\tau^\star}-\theta_n'(\tau^\star)}{w(\tau^\star)}.
\end{eqnarray*}
Since $w(\tau^\star)>0$, we obtain
$$\delta(\tau^\star)=\sign\{ (\tau w)'-\theta'_n(\tau)_{|\tau=\tau^\star}\}\neq 0,$$ which completes the proof.\end{proof}

\begin{remark}\label{delta}
$\delta(\tau^\star)$  reveals the direction along which the complex conjugate eigenvalues $\lambda=\alpha(\tau)\pm i\beta(\tau)$ cross the imaginary axis as $\tau$ increases through $\tau=\tau^\star$. If $\delta(\tau^\star)=1$ (resp. $\delta(\tau^\star)=-1$), $\lambda$  crosses the imaginary axis from left to right (resp. from right to left).
If $\delta(\tau^\star)\neq 0$, the graphs of two functions $\theta=\tau w(\tau)$ and $\theta=\theta_n(\tau)$ intersect transversely at $\tau=\tau^\star$. Otherwise, the two graphs are tangent.
\end{remark}

It is technically difficult to characterize the number of intersections of $\theta=\tau w(\tau)$ and $\theta=\theta_n(\tau)$ or the roots of $S_n(\tau)=0$, $n\in\mathbb{N}$. There may be no roots or several roots.
\begin{definition}
Let $\chi(n)$ be the number of roots of $S_n(\tau)=0$ on $(0, \bar{\tau})$, where $n\in \mathbb{N}$.
\end{definition}

The following hypothesis is necessary for the occurrence of Hopf bifurcation.

\begin{enumerate}
    \item [(H1)]Suppose that $S_n(\tau)=0$ has finitely many roots on $(0, \bar{\tau})$, i.e., $0\leq \chi(n)<\infty, \forall n\in \mathbb{N}$, and these $\chi(n)$ roots are given and arranged in the following order $$\tau_n^{(1)}<\tau_n^{(2)}<\cdots<\tau_n^{(i)}<\cdots \tau_n^{(\chi(n))},$$
with $\delta(\tau_n^{(i)})=\pm 1$ $(i=1, 2, \cdots, \chi(n))$.
\end{enumerate}

\begin{remark}
    Let $\Lambda=\{j\in\mathbb{N}|\chi(j)\neq 0\}$. If $\chi(n)=2$ for all $n\in\Lambda$, then we denote $\tau_n^{(1)}$ and $\tau_n^{(2)}$ as $\tau_n^{-}$ and $\tau_n^{+}$, respectively.
\end{remark}

\noindent {\it Example 1.} Two sets of parameters are chosen to show the distribution of roots of $S_n(\tau)=0$.

\vspace{-\topsep}
\begin{enumerate}

\item [(a)] Choose $K= 1, r=30, m =1, c=4, d=0.1, a=1$. See Fig.\ref{sn56}(a). Here $K<K_0=\frac{41}{39}$, and we have  $\chi(0)=\chi(1)=\chi(2)=\chi(3)=2$, and $\chi(n)=0$ for $n\geq 4$. The critical Hopf bifurcation values are
\begin{equation*}
\left.\begin{array}{llll}
\tau_0^-=0.013562, & \tau_0^+=23.67336, & \tau_1^-=3.8062, & \tau_1^+=21.49988,\\
   \tau_2^-=7.81234, & \tau_2^+=19.55209, & \tau_3^-=12.7067, & \tau_3^+=17.10348.
\end{array}\right.
\end{equation*}

\item [(b)] Choose $K= 20, r=10, m =1, c=4, d=0.1, a=5$.  See Fig.\ref{sn56}(b). Here $K>K_0=\frac{205}{39}$, and we have $\chi(0)=1$, $\chi(1)=2$, $\chi(2)=4$, and $\chi(n)=0$ for $n\geq 3$. The critical Hopf bifurcation values are
\begin{equation*}
\left.\begin{array}{llll}
\tau_0^{(1)}=32.42808, &
\tau_1^{(1)}=6.2049, & \tau_1^{(2)}=32.22374, &  \\
\tau_2^{(1)}=13.96008, & \tau_2^{(2)}=24.85518, & \tau_2^{(3)}=30.21286, & \tau_2^{(4)}=32.00694.
\end{array}\right.
\end{equation*}
\end{enumerate}
\vspace{-\topsep}

\begin{figure}[H]
\begin{center}
   \subfigure[$K<K_0$]{\includegraphics[angle=0, width=0.485\textwidth]{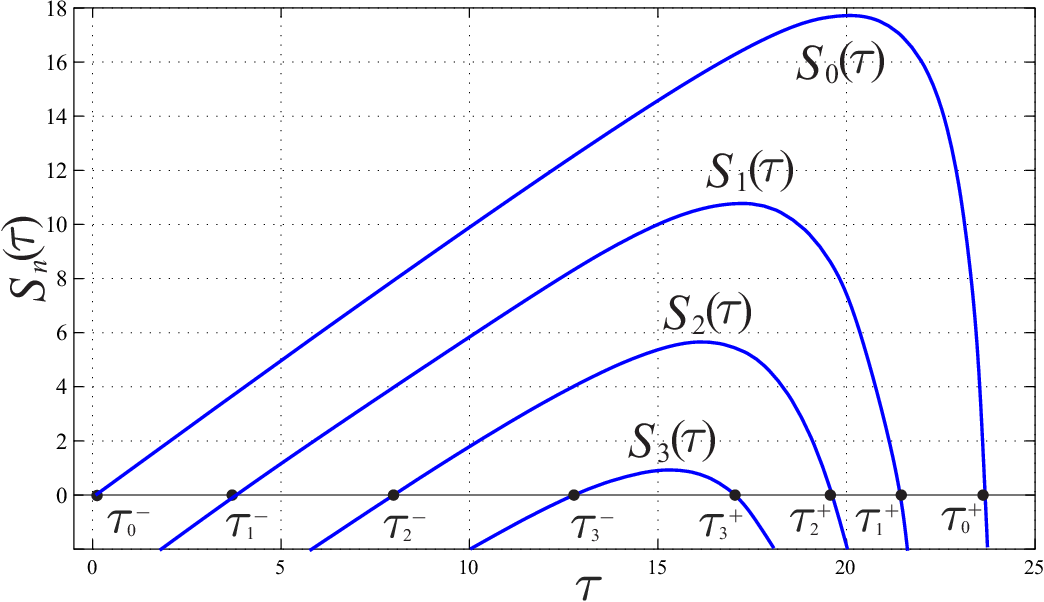}} \quad
   \subfigure[$K>K_0$]{\includegraphics[angle=0, width=0.47\textwidth]{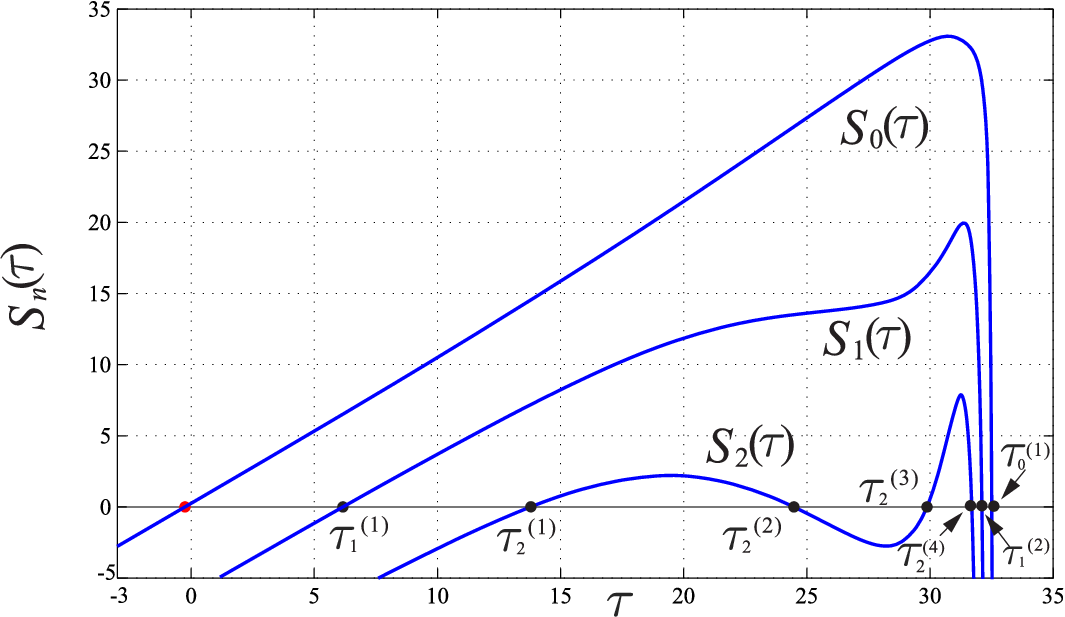}}
    \caption{Left panel: Roots of $S_n(\tau)=0$ for the parameters given in Example 1(a), where $K<K_0$. Right panel: Roots of $S_n(\tau)=0$ for the parameters given in Example 1(b), where $K>K_0$.}\label{sn56}
\end{center}
\end{figure}

\begin{corollary}\label{chopf}

Assume (H1), then Hopf bifurcation occurs at a sequence of critical values $\tau=\tau_n^{(i)}$ and there exists a family of non-constant periodic solutions in a neighborhood of $E^\ast$ for $\tau$ in a small neighborhood of $\tau=\tau_n^{(i)}$.
\end{corollary}

The following is an attempt to study the roots of $S_n(\tau)=0$ for certain special cases.

\begin{proposition}\label{2n_thm}
Let
$M=\displaystyle\max_{\tau\in[0, \bar{\tau}]}\{\tau w(\tau)\}$. Assume that
$$K<K_0, \ \text{and}\ \ (2N+1)\pi<M< (2N+3)\pi $$
for some $N \in \mathbb{N}$. Then there are three cases.

(1) If $n\leq N$, $\theta_n(\tau)$ intersects $\tau w(\tau)$ at least twice.

(2) If $n=N+1$, $\theta_n(\tau)$ and $\tau w(\tau)$ intersect either at least twice, or once tangentially, or they never intersect.

(3) If $n\geq N+2$, $\theta_n(\tau)$ and $\tau w(\tau)$ never intersect in $(0,\bar{\tau})$.






\end{proposition}

\begin{proof}

By Lemma \ref{theta12}, $\theta_n(\tau)$ ranges between $[\zeta+2n\pi, (2n+1)\pi]$. From statement (1) of Lemma \ref{theta12}, we have $\tau w(\tau)\big|_{\tau=0}=0$ and $ \tau w(\tau)\big|_{\tau=\bar{\tau}}=0$ and $M>0$. For the first case, from the intermediate value theorem, it follows that $\theta_n(\tau)$ must intersect $\tau w(\tau)$ at least twice since $\theta_n(0)\geq \zeta+2n\pi>0 $,  $M> (2N+1)\pi > \max_{\tau\in[0, \bar{\tau}]}\theta_n(\tau)$,
and $\theta_n(\bar{\tau})>0$. For the third case, we have $n\geq N+2$, hence $\min_{\tau\in[0, \bar{\tau}]}\theta_n(\tau)> 2n\pi \geq 2(N+2)\pi > M$ and the result follows. The second case is trivial, as all possibilities might hold here.
\end{proof}

\subsection{Local dynamics and bifurcations involving \texorpdfstring{$E_K$}{} }

\begin{theorem}\label{boundary}

(1) $E_K$ is asymptotically stable for $\tau>\tau_{max}$, and unstable for $0<\tau<\tau_{max}$.

(2) System \eqref{modelH2} undergoes a transcritical bifurcation involving $E_K$ and $E^\ast$ at $\tau=\tau_{max}.$
\end{theorem}

\begin{proof}
(1) The characteristics equation at $E_K$ has the following form
\begin{equation*}
(\lambda+r)(\lambda+d-\frac{ce^{-d\tau}mK}{a+K}e^{-\lambda \tau})=0.
\end{equation*}
Therefore, $\lambda=-r$ is an eigenvalue, and the other eigenvalues are the roots of the equation

\begin{equation}
    Q(\lambda):=\lambda+d-\frac{ce^{-d\tau}mK}{a+K}e^{-\lambda \tau}=0.
\end{equation}

As one did for $E^\ast$, an analogous analysis can be performed to investigate the stability of $E_K$. Here, using Theorem 4.7(a) in Chapter 4 of \cite{Smith2010} we conclude that $E_K$ is asymptotically stable if $\tau>\tau_{max}$. If $0<\tau<\tau_{max}$, then $Q(0)<0$. Note that $\displaystyle \lim_{\lambda\rightarrow\infty, \lambda\in \mathbb{R}}Q(\lambda)=\infty,$ by continuity of $Q(\lambda)$, there exists a positive real root. Hence, $E_K$ is always unstable.


(2) If $\tau=\tau_{max}$,  then $E_K=E^\ast$ and $Q(0)=0, Q'(0)=1+d\tau_{max}\neq 0$. Hence, $\lambda=0$ is a simple eigenvalue.

Let $\hat{\lambda}(\tau)$ be a root of $Q(\lambda)=0$ such that $\hat{\lambda}(\tau_{max})=0$. To verify that $E_K$ changes stability at $\tau= \tau_{max}$, we compute
$$\frac{d}{d\tau}Re(\hat{\lambda})\bigg|_{\tau=\tau_{max}}=-\frac{ cmdKe^{-d\tau}}{a+K+\tau cmK e^{-d\tau}}\bigg|_{\tau=\tau_{max}} <0.$$
As a consequence, $E_K$ changes stability from unstable to stable. When investigating the stability switch of $E^*$ at $\tau=\tau_{max}$, let $\tilde{\lambda}(\tau)$ be a root of $P(\lambda, \tau)=0$ such that $\tilde{\lambda}(\tau_{max})=0$ at $\tau=\tau_{max}$, then
$$\frac{d}{d\tau}Re(\tilde{\lambda})\bigg|_{\tau=\tau_{max}}=\frac{L'(\tau)-dH'(\tau)}{H+\tau L}\bigg|_{\tau=\tau_{max}}=\frac{d^2}{\ln\left(\frac{cmK}{d\,\left(a+K\right)}\right)+1}>0.
$$
Hence, the real part of $\tilde{\lambda}$ changes sign at $\tau=\tau_{max}$ from negative to positive. This implies that $E^*$ moves from the first quadrant to the fourth quadrant, and changes stability from stable to unstable. Thus, $E_K$ and $E^\ast$ exchange stability at $\tau=\tau_{max}$, where they merge in a transcritical bifurcation.
\end{proof}


\section{Global Hopf bifurcation and continuity of branches}


By the local Hopf bifurcation Theorem \ref{lhopf} and Corollary \ref{chopf}, the periodic solutions bifurcating from $E^\ast$ exist in a small neighborhood of a sequence of critical values $\tau=\tau_n^{(i)}$ where $i=1, 2, \cdots, \chi(n)$ and $n\in\mathbb{N}$.  In this section,
we study the evolution of periodic solutions as $\tau$ varies within $(0, \tau_{max})$. The global Hopf bifurcation theory developed by Wu is implemented (see Theorem 3.3 of \cite{Wu1998}), and connected components of non-trivial periodic solutions in Fuller's space are considered. By estimating the period of the period solution and studying the $\tau$-interval in which the period solution exists, we can describe the onset and termination of Hopf bifurcation branches under certain conditions.

Let $\tilde{x}(t)=x(t\tau)$, $\tilde{y}(t)=y(t\tau)$, then system \eqref{modelH2} can be rewritten as
\begin{equation}\label{model global}
    \dot{z}=F(z_t,\tau,T), \ (t, \tau, T)\in \mathbb{R}_+\times (0, \tau_{max})\times \mathbb{R}_+,
\end{equation}
where $$z(t)=\begin{bmatrix}
    \tilde{x}(t)\\
    \tilde{y}(t)
\end{bmatrix}, \,
   F(z_t,\tau,T)= \begin{bmatrix}
        \tau r\tilde{x}(t)(1-\frac{\tilde{x}(t)}{K})-\tau \frac{m\tilde{x}(t)\tilde{y}(t)}{a+\tilde{x}(t)}\\
            -\tau d \tilde{y}(t)+\tau e^{-d\tau}\frac{cm\tilde{x}(t-1)\tilde{y}(t-1)}{a+\tilde{x}(t-1)}
    \end{bmatrix}.
$$
Let $z_t(s)=z(t+s)$ for $s\in [-1,0]$, so that $z_t\in C([-1, 0], \mathbb{R}^2_+)$. Hence, the domain of $F(z_t,\tau, T)$ is $ C([-1, 0], \mathbb{R}^2_+)  \times (0,\tau_{max}) \times \mathbb{R}_+$. By identifying
 the subspace of $C([-1, 0], \mathbb{R}^2_+)$ consisting of all constant mapping with $\mathbb{R}^2_+$, we obtain a restricted mapping
 $$\tilde{F}(z,\tau,T):=F(z_t,\tau,T)_{|\mathbb{R}_+^2\times (0,\tau_{max})\times \mathbb{R}_+}=
\begin{bmatrix}
        \tau r\tilde{x}(1-\frac{\tilde{x}}{K})-\tau\frac{m\tilde{x}\tilde{y}}{a+\tilde{x}}\\
            -\tau d\tilde{y}+\tau e^{-d\tau}\frac{cm\tilde{x}\tilde{y}}{a+\tilde{x}}
    \end{bmatrix}.
$$
Then $\tilde{F}$ is twice continuously differentiable, and condition (A1) of Theorem 3.3 of \cite{Wu1998} is fulfilled.  

Let $\mathscr{N}(F)=\{(z,\tau,T)| \tilde{F}(z,\tau,T)=0\}$ be the set of stationary solutions of system \eqref{model global}. From our analysis in Section 3, we know
$$\mathscr{N}(F)=\{(z, \tau, T)| z\in \{E_0, E_K, E^\ast\}, (\tau, T)\in (0, \tau_{max})\times \mathbb{R}_+\}.$$

It is apparent that the function $F(z_t, \tau, T)$ is continuously differentiable with respect to its first argument $z_t$ and the characteristics matrix function at the stationary solution $(z, \tau, T)\in\mathscr{N}(F)$ is
$$\Delta_{(z,\tau,T)}(\lambda)=\lambda Id-DF(z, \tau, T)(e^{\lambda\cdot} Id).$$
Here, $DF(z, \tau, T)$ is the complexification of the derivative of $F(z_t, \tau, T)$ with respect to $z_t$ evaluated at $(z, \tau, T)\in \mathscr{N}(F)$. Since $\Delta_{(z,\tau,T)}$ is continuous,
condition (A3) of Theorem 3.3 of \cite{Wu1998} is fulfilled. 

Furthermore, by Lemma \ref{origin} and Theorem \ref{boundary}, we establish that $\lambda=0$ is not an eigenvalue of $E_0$, $E_K$ or $E^\ast$ for $\tau\in (0, \tau_{max})$, meaning $\text{det}(DF(z, \tau, T))\neq 0$ for $(z, \tau, T)\in \mathscr{N}(F)$. Thus, $DF(z, \tau, T)$ is isomorphic to $\mathbb{R}^2$ at each stationary solution, and condition (A2) of Theorem 3.3 of \cite{Wu1998} is fulfilled. 

For the global Hopf bifurcation, by  Lemma \ref{positivity} and Theorem \ref{lhopf}, $(E^*, \tau, T)\in \mathscr{N}(F)$ is the only stationary solution under consideration. Assume condition (H1) is satisfied, i.e., $S_n(\tau)=0$ has $\chi(n)$ roots on $(0, \bar{\tau})$, where $\chi(n)<\infty, \forall n\in \mathbb{N}$, and these $\chi(n)$ roots are given and arranged in the following order $$\tau_n^{(1)}<\tau_n^{(2)}<\cdots<\tau_n^{(i)}<\cdots \tau_n^{(\chi(n))}.$$
    Furthermore,  $\delta(\tau_n^{(i)})=\pm 1$ $(i=1, 2, \cdots, \chi(n))$.

By the definition of a center for a stationary solution \cite{Wu1998}, the points
\begin{equation}\label{center}
  (E^\ast, \tau_n^{(i)}, \frac{2\pi}{\tau_n^{(i)}w(\tau_n^{(i)})})=(E^\ast, \tau_n^{(i)}, \frac{2\pi}{\theta_n(\tau_n^{(i)})}), \ n\in \mathbb{N}, \ i=1, 2,\cdots, \chi(n)  \end{equation}
are centers of system \eqref{model global}. For each center, it is the only center in its small neighborhood, and it only has finitely many purely imaginary characteristic values, whose imaginary part takes the form $$\tau_n^{(i)}w(\tau_n^{(i)}=i\frac{2\pi}{2\pi/(\tau_n^{(i)}w(\tau_n^{(i)}))}=\frac{2\pi}{2\pi/(\theta_n(\tau_n^{(i)}))}.$$
Here $T_n^{(i)}:=2\pi/(\tau_n^{(i)}w(\tau_n^{(i)}))=2\pi/(\theta_n(\tau_n^{(i)}))$ is the minimal period of periodic solution bifurcating at $\tau_n^{(i)}$ for the scaled model \eqref{model global}.
Therefore, $(E^*,\tau_n^{(i)}, \frac{2\pi}{\theta(\tau_n^{(i)})})$ given in \eqref{center} are isolated centers. 
Moreover, the first crossing number of each isolated center $\gamma_1(E^*,\tau_n^{(i)}, \frac{2\pi}{\theta(\tau_n^{(i)})})$ is well defined, and
$$\gamma_1(E^*,\tau_n^{(i)}, \frac{2\pi}{\theta(\tau_n^{(i)})})=-\sign\left(\frac{d}{d\tau}Re(\lambda)\big|_{\tau=\tau_n^{(i)}}\right)=-\sign(S'(\tau_n^{(i)}))=-\delta(\tau_n^{(i)})=\pm 1.$$
Hence, condition (A4) of Theorem 3.3 of \cite{Wu1998} is fulfilled.

Define the closed set $\Sigma(F)$ in $C([-1, 0], \mathbb{R}^2_+)\times (0, \tau_{max})\times \mathbb{R}_+$ as follows:
$$\Sigma(F)=Cl\{(z,\tau,T) \in C([-1, 0], \mathbb{R}^2_+)\times (0, \tau_{max})\times \mathbb{R}_+|z\  \text{is a $T$-periodic solution of system \eqref{model global}}\}.$$ Let $\mathscr{C}(E^*,\tau_n^{(i)}, \frac{2\pi}{\theta(\tau_n^{(i)})})$ be the connected component of $(E^*,\tau_n^{(i)}, \frac{2\pi}{\theta(\tau_n^{(i)})})$ in $\Sigma(F).$

By ensuring that conditions (A1)-(A4) of Theorem 3.3 in \cite{Wu1998} hold, we can now apply the global Hopf bifurcation theorem to system \eqref{model global}.

\begin{theorem}\label{global-hopf-theorem}
Consider the system \eqref{model global} with (H1). One of the following two statements holds.

(i) $\mathscr{C}(E^*,\tau_n^{(i)}, \frac{2\pi}{\theta(\tau_n^{(i)})})$ is unbounded;

(ii) $\mathscr{C}(E^*,\tau_n^{(i)}, \frac{2\pi}{\theta(\tau_n^{(i)})})$ is bounded, the intersection $\mathscr{C}(E^*,\tau_n^{(i)}, \frac{2\pi}{\theta(\tau_n^{(i)})}) \cap \mathscr{N}(F)$ is finite, and
$$\sum_{(z,\tau,T)\in \mathscr{C}(E^*,\tau_n^{(i)}, \frac{2\pi}{\theta(\tau_n^{(i)})})\cap \mathscr{N}(F)}\gamma_1(z,\tau,T)=0.$$

Here $i=1, 2,\cdots, \chi(n)$ and $n\in\mathbb{N}$.
\end{theorem}

In the following, we analyze the boundedness and connections of $\mathscr{C}(E^*,\tau_n^{(i)}, \frac{2\pi}{\theta(\tau_n^{(i)})})$.

\begin{lemma}\label{pbound}
Let $\Gamma:=\{(\xi_1(t), \xi_2(t)): t\in [0, \infty)\}$ be any non-constant periodic solution of system \eqref{modelH2} with initial condition $\phi\in \mathbf{C\times C}$. Then $0< \xi_1(t)\leq K, 0<\xi_2(t)\leq cK,$ for all $t\geq 0$.
\end{lemma}

\begin{proof}
By Lemma \ref{positivity}, $\xi_1$ and $\xi_2$ are positive for all $t\geq 0$. By Lemma \ref{boundedness}, we have $\displaystyle \limsup_{t\rightarrow\infty} x(t)\leq K$. We assert that $\xi_1(t)\leq K$ for all $t\geq 0$. If not, then there exists $t_1>0$ such that $\xi_1(t_1)>K$. Then
$$K<\xi_1(t_1)=\lim_{n\rightarrow\infty}\xi_1(t_1+nT)\leq \limsup_{t\rightarrow\infty} \xi_1(t)\leq K,$$
which is a contradiction. By a similar argument, we conclude that $\xi_2(t)\leq cK$ for all $t\geq 0$.
\end{proof}

\begin{lemma}\label{no-period-one}
If $K<K_0$, system \eqref{model global} has no periodic solutions of period 1.
\end{lemma}

\begin{proof}
At a periodic solution with period 1, we have $\tilde{x}(t-1)=\tilde{x}(t)$ and $\tilde{y}(t-1)=\tilde{y}(t)$. To show that system \eqref{model global} has no periodic solution of period 1, we replace $\tilde{x}(t-1)$ and $\tilde{y}(t-1)$ with $\tilde{x}(t)$ and $\tilde{y}(t)$ respectively, reducing system \eqref{model global} to the ODE system \begin{equation}\label{ODE period 2}
      \begin{cases}
    \frac{d\tilde{x}}{dt}=\tau r\tilde{x}(1-\frac{\tilde{x}}{K})-\tau \frac{m\tilde{x}\tilde{y}}{a+\tilde{x}},\\
    \frac{d\tilde{y}}{dt}=\tau \tilde{y}(-d+ce^{-d\tau}\frac{m\tilde{x}}{a+\tilde{x}}).
\end{cases}
\end{equation}
By Lemma \ref{ODEsystemdynamics} and substituting $c$ with $ce^{-d\tau}$, we obtain that system \eqref{ODE period 2} does not have a periodic solution when $K< \frac{a(cme^{-d\tau}+d)}{cme^{-d\tau}-d}$. Notice that $K_0=\frac{a(cm+d)}{cm-d}<\frac{a(cme^{-d\tau}+d)}{cme^{-d\tau}-d}$, and therefore system \eqref{ODE period 2} has no periodic solutions if $K<K_0$. Consequently, system \eqref{model global} has no periodic solutions of period 1.
\end{proof}

We have the estimate of $T_n^{(i)}$, the minimal period of periodic solution bifurcating at $\tau_n^{(i)}$ for the scaled model \eqref{model global} below.

\begin{lemma}\label{tinterval}
If $K<K_0$, 
then $T_n^{(i)}\in(\frac{1}{n+1},\frac{1}{n})$ for $i=1,2,\cdots, \chi(n)$ and $n=1, 2,\cdots$. \end{lemma}

\begin{proof}
By Lemma \ref{theta12} we have $2n\pi+\zeta \leq\theta_n(\tau_n^{(i)})\leq (2n+1)\pi$ for all $\tau \in [0,\bar{\tau}]$. Then
$$\frac{1}{n+1}=\frac{2\pi}{(2n+2)\pi}<\frac{2\pi}{(2n+1)\pi}\leq T_n^{(i)}\leq \frac{2\pi}{2n\pi+\zeta}<\frac{1}{n}.$$
\end{proof}

\begin{theorem} \label{no-extended-branches}
For system \eqref{model global}, $E_K$ is globally asymptotically stable for $\tau\geq\tau_{max}$, and periodic solutions exist only on $(0,\tau_{max})$.
\end{theorem}

\begin{proof}
Consider the Lyapunov functional
\begin{equation*}
    V(\tilde{x}, \tilde{y})=\tilde{x}-K-K\ln{\frac{\tilde{x}}{K}}+\frac{mK}{da}\tilde{y}+\tau\frac{cm^2Ke^{-d\tau}}{da}\int_{t-1}^{t}\frac{\tilde{x}(s)\tilde{y}(s)}{a+\tilde{x}(s)}ds.
\end{equation*}
By Lemma \ref{positivity}, $V(\tilde{x}, \tilde{y})$ is well-defined. Furthermore, $V(\tilde{x}, \tilde{y})\geq 0$ for all $\tilde{x}>0, \tilde{y}\geq 0$, and $V(\tilde{x}, \tilde{y})=0$ if and only if $(\tilde{x}, \tilde{y})=(K, 0)$, we compute the total derivative along system \eqref{model global} with respect to $t$:
\begin{eqnarray*}
\dot{V}(\tilde{x}, \tilde{y})&=&r\tau \tilde{x}(t)(1-\frac{\tilde{x}(t)}{K})-\tau\frac{m\tilde{x}(t)\tilde{y}(t)}{a+\tilde{x}(t)}-K\tau\left(r(1-\frac{\tilde{x}(t)}{K})-\frac{m\tilde{y}(t)}{a+\tilde{x}(t)}\right)\\
& &+\tau\frac{cm^2Ke^{-d\tau}}{da}\cdot\frac{\tilde{x}(t-1)\tilde{y}(t-1)}{a+\tilde{x}(t-1)}-\tau\frac{mK}{a}\tilde{y}(t)\\
& &+\tau\frac{cm^2Ke^{-d\tau}}{da}\cdot\frac{\tilde{x}(t)\tilde{y}(t)}{a+\tilde{x}(t)}-\tau \frac{cm^2Ke^{-d\tau}}{da}\cdot\frac{\tilde{x}(t-1)\tilde{y}(t-1)}{a+\tilde{x}(t-1)}\\
&=&\tau \left[-\frac{r}{K}(K-\tilde{x})^2 +\tilde{y}\left(\frac{-m\tilde{x}+Km}{a+\tilde{x}}-\frac{mK}{a}+\frac{cm^2Ke^{-d\tau}\tilde{x}}{da(a+\tilde{x})}\right)\right].
\end{eqnarray*}
For $\tau\geq\tau_{max}$, we obtain
\begin{eqnarray*}
\dot{V}(\tilde{x}, \tilde{y})&\leq & \tau \left[-\frac{r}{K}(K-\tilde{x})^2 +\tilde{y}\left(\frac{-m\tilde{x}+Km}{a+\tilde{x}}-\frac{mK}{a}+\frac{cm^2Ke^{-d\tau_{max}}\tilde{x}}{da(a+\tilde{x})}\right)\right]\\
&=&-\frac{\tau r}{K}(K-\tilde{x})^2\leq 0.
\end{eqnarray*}
The last equality holds if and only if $\tilde{x}=K$ and $\tilde{y}\geq0$. From the first equation of system \eqref{model global}, the largest invariant set of $\{(K, \tilde{y})|\tilde{y}\geq 0\}$ is the singleton $\{(K, 0)\}$.  By Lyapunov-LaSalle's invariance principle \cite{Kuang1993}, $(K,0)$ is globally asymptotically stable for $\tau\geq \tau_{max}$.
\end{proof}

\begin{figure}[!ht]
\begin{center}
\includegraphics[angle=0, width=0.8\textwidth]{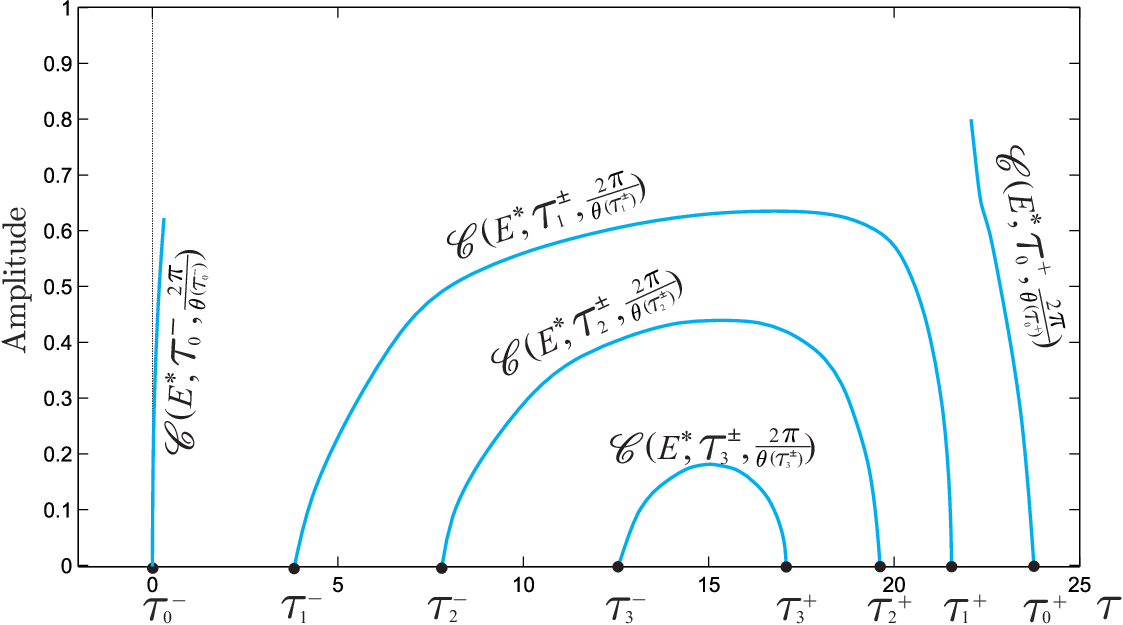}
\caption{\small{Global Hopf bifurcation and connected components for $K<K_0$. The parameters and values of $\tau_n^\pm (n=0, 1, 2, 3)$ are provided in Example 1(a).}}\label{fourbranch}
\end{center}
\end{figure}

\begin{theorem}\label{smallK}
Consider system \eqref{model global}, under the assumption (H1) with $K<K_0$. Then the following two statements hold.

\begin{itemize}
    \item[(1)] All connected components $\mathscr{C}(E^*,\tau_n^{(i)}, \frac{2\pi}{\theta(\tau_n^{(i)})})$ are bounded for all $n=1, 2, \cdots$.

    \item[(2)] If $\chi(n)=2$, denote $\tau_n^{(1)}$ and $\tau_n^{(2)}$ as $\tau_n^{-}$ and $\tau_n^{+}$, respectively. Then the connected components $\mathscr{C}(E^*,\tau_n^{-}, \frac{2\pi}{\theta(\tau_n^{-})})$ and $\mathscr{C}(E^*,\tau_n^{+}, \frac{2\pi}{\theta(\tau_n^{+})})$ coincide and are nested. The connected component
   $\mathscr{C}(E^*,\tau_n^{\pm}, \frac{2\pi}{\theta(\tau_n^{\pm})})$ connects exactly two bifurcation points $\tau_n^-$ and $\tau_n^+$, where $n=1, 2, \cdots$.
        \end{itemize}
\end{theorem}

\begin{proof}
(1). By Lemma \ref{pbound}, all periodic solutions of system \eqref{model global} are uniformly bounded.  Therefore, the projection of $\mathscr{C}(E^*,\tau_n^{(i)}, \frac{2\pi}{\theta(\tau_n^{(i)})})$ onto $C([-1, 0], \mathbb{R}^2_+)$ is bounded.

By Lemma \ref{ODEsystemdynamics} and Lemma \ref{no-extended-branches}, no periodic solution exists for $\tau=0$ or $\tau\geq \tau_{max}$. Consequently, the connected components do not extend beyond the finite interval $(0,\tau_{max})$. Hence, the projection of $\mathscr{C}(E^*,\tau_n^{(i)}, \frac{2\pi}{\theta(\tau_n^{(i)})})$ onto $\tau$-space is bounded.

Furthermore, by Lemma \ref{no-period-one},
there are no periodic solutions of period 1, and consequently, none with period $\frac{1}{n}, n=1,2,\cdots$. By the continuity of the connected component, the projection of $\mathscr{C}(E^*,\tau_n^{(i)}, \frac{2\pi}{\theta(\tau_n^{(i)})})$ onto the $T$-space is a subset of the finite interval $(\frac{1}{n+1}, \frac{1}{n})$.  Hence, all connected components $\mathscr{C}(E^*,\tau_n^{(i)}, \frac{2\pi}{\theta(\tau_n^{(i)})})$ are bounded in Fuller's space, where $n=1, 2, \cdots$.

(2). By Lemma \ref{tinterval}, the minimal period of the periodic solution bifurcating at $\tau_n^{\pm}$ is
$T_n^{\pm}=\frac{2\pi}{\theta(\tau_n^{\pm})}\in (\frac{1}{n+1}, \frac{1}{n})$. By Lemma \ref{no-period-one}, the projection of $\mathscr{C}(E^*,\tau_n^{\pm}, \frac{2\pi}{\theta(\tau_n^{\pm})})$ onto the $T$-space is a subset of the finite interval $(\frac{1}{n+1}, \frac{1}{n})$. Moreover, $\mathscr{C}(E^*,\tau_n^{-}, \frac{2\pi}{\theta(\tau_n^{-})})$ and $\mathscr{C}(E^*,\tau_n^{+}, \frac{2\pi}{\theta(\tau_n^{+})})$ are the only two connected components whose third arguments belong to $(\frac{1}{n+1}, \frac{1}{n})$.

By Theorem \ref{global-hopf-theorem}, the first crossing number of $(E^*,\tau_n^{\pm}, \frac{2\pi}{\theta(\tau_n^{\pm})})$ on $\mathscr{C}(E^*,\tau_n^{\pm}, \frac{2\pi}{\theta(\tau_n^{\pm})})$ satisfies
$$\sum_{(z,\tau,T)\in \mathscr{C}(E^*,\tau_n^{\pm}, \frac{2\pi}{\theta(\tau_n^{\pm})})\cap \mathscr{N}(F)}\gamma_1(z,\tau,T)=0.$$
Note that $$\gamma_1(E^*,\tau_n^{-}, \frac{2\pi}{\theta(\tau_n^-)})=-1, \ \text{and}\ \ \gamma_1(E^*,\tau_n^{+}, \frac{2\pi}{\theta(\tau_n^+)})=1.$$ Hence, $\mathscr{C}(E^*,\tau_n^{-}, \frac{2\pi}{\theta(\tau_n^{-})})$ and $\mathscr{C}(E^*,\tau_n^{+}, \frac{2\pi}{\theta(\tau_n^{+})})$ must coincide, initiating at $\tau=\tau_n^-$ and terminating at $\tau_n^+$ in the $\tau$-space (or vice versa). Since $\tau_j^-<\tau_i^-<\tau_i^+<\tau_j^+$ for $j<i$, the components $\mathscr{C}(E^*,\tau_n^{\pm}, \frac{2\pi}{\theta(\tau_n^{\pm})})$ are nested, where $n=1, 2, \cdots$.
\end{proof}

If $\chi(0)=2$ and assume the projection of $\mathscr{C}(E^*,\tau_0^{-}, \frac{2\pi}{\theta(\tau_0^{-})})$ and $\mathscr{C}(E^*,\tau_0^{+}, \frac{2\pi}{\theta(\tau_0^{+})})$ onto the $T$-space are bounded, by Theorem \ref{global-hopf-theorem} and a similar argument as that of Theorem \ref{smallK}, we obtain the following result.

\begin{theorem}\label{first branch} Let $K<K_0$ and $\chi(0)=2$.
    If the projection of $\mathscr{C}(E^*,\tau_0^{-}, \frac{2\pi}{\theta(\tau_0^{-})})$ and $\mathscr{C}(E^*,\tau_0^{+}, \frac{2\pi}{\theta(\tau_0^{+})})$ onto the $T$-space are bounded, then they coincide and connect two bifurcation points $\tau_0^-$ and $\tau_0^+$.
\end{theorem}

It is important to note that the nonexistence of periodic solutions in the ODE system \eqref{ODE model} (i.e., system \eqref{modelH2} plays a crucial role in establishing the boundedness and connectedness of the connected components, as demonstrated in the proof of Theorem \ref{smallK}. However, when $K>K_0$, system \eqref{modelH2} admits a periodic solution at $\tau=0$. Consequently, system \eqref{model global} also possesses a periodic solution of period 1. As a result, we lose the ability to estimate the period of the connected components in the $T$-space, and whether $\mathscr{C}(E^*,\tau_n^{-}, \frac{2\pi}{\theta(\tau_n^{-})})$ and $\mathscr{C}(E^*,\tau_n^{+}, \frac{2\pi}{\theta(\tau_n^{+})})$ coincide remains unclear if $\chi(n)=2$ (see Fig. \ref{fourbranch} for these two connected components. It is not clear whether they are connected due to the computational challenge).

Nevertheless, we will explore the relations between the connected components for $K>K_0$ through numerical analysis using {\it DDE-BifTool}. Instead of analyzing $\mathscr{C}(E^*,\tau_n^{(i)}, \frac{2\pi}{\theta(\tau_n^{(i)})})$ for system \eqref{model global}, we consider the connected components $\mathscr{C}(E^*,\tau_n^{(i)}, \frac{2\pi}{w(\tau_n^{(i)})})$ for system \eqref{modelH2}, since  a periodic solution exists at $\tau=0$.

\begin{figure}[!ht]
\begin{center}
\includegraphics[angle=0, width=0.96\textwidth]{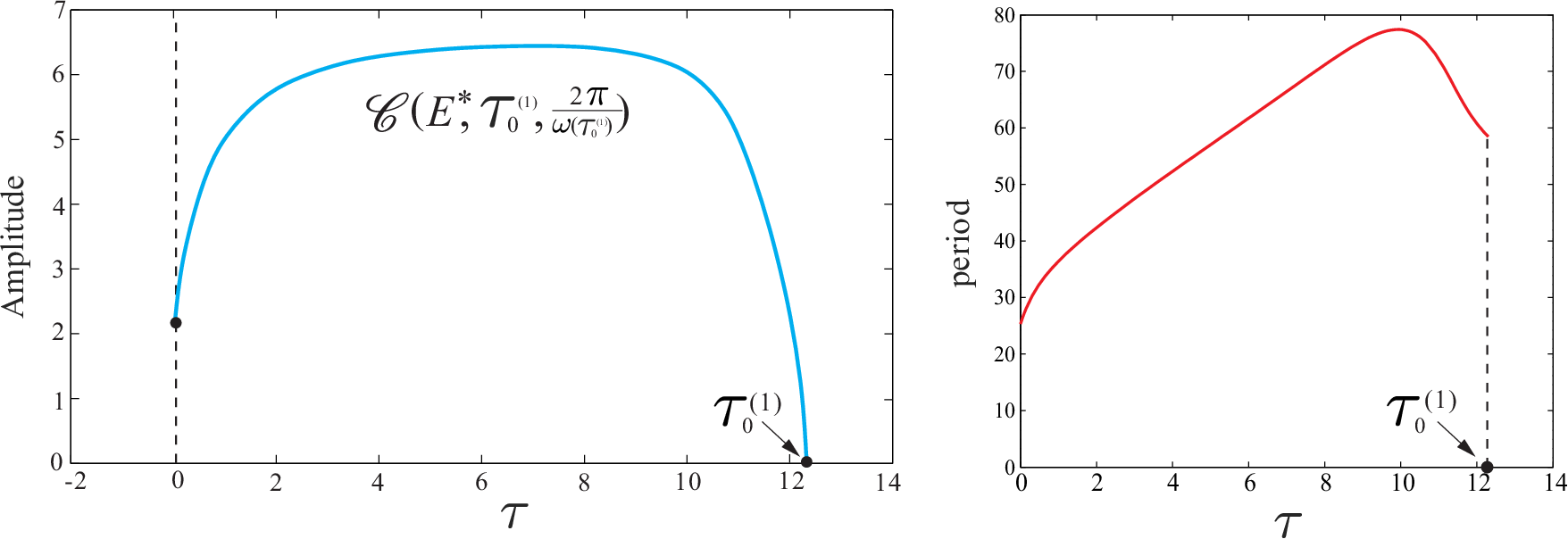}
\caption{\small{Global Hopf bifurcation and a unique connected component $\mathscr{C}(E^*,\tau_0^{(1)}, 2\pi/w(\tau_0^{(1)}))$ for $K>K_0$.
The left panel shows the amplitude and the right panel displays the period of the periodic solution for $\tau\in[0, \tau_0^{(1)}]$. The parameter are $K=7, r=1, m=1, c=1, d=0.1, a=5$, and $K>K_0=6.11$.
}}\label{onebranch}
\end{center}
\end{figure}

\setlength{\leftmargini}{0.3cm}
\begin{itemize}
    \item {\it A unique connected component and the limit cycle of the ODE model \eqref{ODE model}.}

If $\chi(0)=1$ and $\chi(n)=0$ for $n\geq 1$, $S_0(\tau)=0$ has a unique root $\tau=\tau_0^{(1)}$ on $[0, \bar{\tau})$ and system \eqref{modelH2} has a unique connected component $\mathscr{C}(E^*,\tau_0^{(1)}, \frac{2\pi}{w(\tau_0^{(1)})})$. When $\tau=0$, $J(E^\ast)$ has a pair of conjugate complex eigenvalues with positive real part, meaning $E^\ast$ is unstable and a stable limit cycle $\Gamma_0$ (i.e, an isolated periodic solution) exists around $E^\ast$.  As $\tau$  increases from $0$ to $\tau_0^{(1)}$, and passes by $\tau_0^{(1)}$,  Lemma \ref{roots} implies that this pair of complex eigenvalues vary continuously, crossing the imaginary axis from right to left at $\tau=\tau_0^{(1)}$ so that $E^\ast$ becomes stable. Meanwhile, the limit cycle $\Gamma_0$ evolves continuously. However, $\Gamma_0$  does not necessarily persist over the entire interval $[0, \tau_0^{(1)})$, as it may vanish due to a bifurcation of periodic solutions at some $\tau<\tau_0^{(1)}$. If no such bifurcation occurs for $\tau\in(0, \tau_0^{(1)})$, then $\Gamma_0$ must disappear through a local Hopf bifurcation at $\tau=\tau_0^{(1)}$, i.e., $(\Gamma_0, 0, \frac{2\pi}{\bar{w}})\in \mathscr{C}(E^*,\tau_0^{(1)}, \frac{2\pi}{w(\tau_0^{(1)})})$, where 
$\bar{w}$ is the frequency of the limit cycle $\Gamma_0$.Thus, {\it the limit cycle of the ODE model \eqref{ODE model} belongs to a connected component of the Hopf bifurcation of the DDE model \eqref{modelH2} in Fuller's space}. As illustrated in Fig. \ref{onebranch}, $\Gamma_0$ varies continuously on $[0, \tau_0^{(1)}]$ and terminates $\tau=\tau_0^{(1)}$ through a Hopf bifurcation.
\end{itemize}

 \begin{figure}[!ht]
\begin{center}
\includegraphics[angle=0, width=0.96\textwidth]{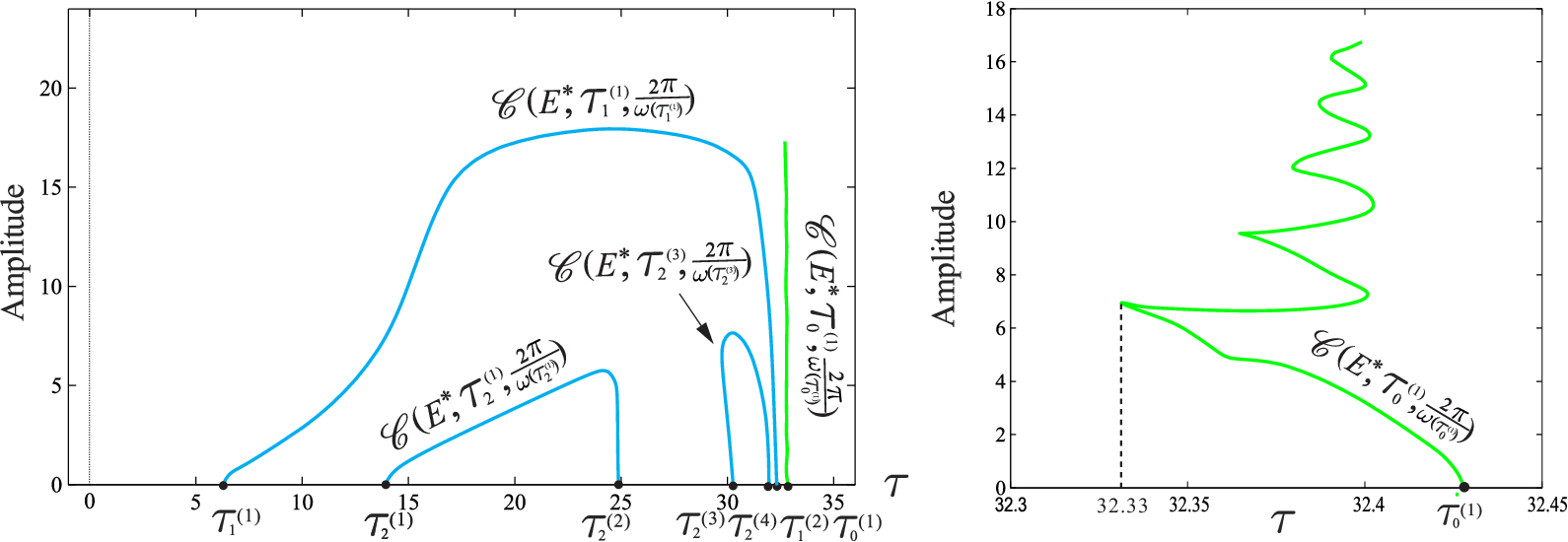}
\caption{\small{Global Hopf bifurcation and multiple connected components for $K>K_0$. The right panel is a magnified view of the left panel near $\tau=\tau_0^{(1)}$. The parameter values and $\tau_n^{(i)} (n=0, 1, 2)$ are specified in Example 1(b).}}\label{threebranchv10}
\end{center}
\end{figure}

\begin{itemize}

\item {\it Multiple connected components. }

If $\chi(n)>2$ for some $n\in\mathbb{N}$, then $S_n(\tau)=0$ has $\chi(n)$ critical values $\tau=\tau_n^{{(i)}}$, leading to $\chi(n)$ connected components $\mathscr{C}(E^*,\tau_n^{(i)}, \frac{2\pi}{\theta(\tau_n^{(i)})})$, where $i=1, 2, \cdots, \chi(n)$.
Studying the relationship between these connected components is technically challenging, regardless of whether $K<K_0$ or $K>K_0$. To illustrate this complexity, we consider the parameter values from Example 1(b), where $\chi(0)=1, \chi(1)=2$, $\chi(2)=4$, and $\chi(n)=0$ for $n\geq 3$. Based on Fig. \ref{threebranchv10} we have the following observations:
\begin{itemize}
\item For $n=0$: $\mathscr{C}(E^*,\tau_0^{(1)}, \frac{2\pi}{w(\tau_0^{(1)})})$ initials at $\tau_0^{(1)}$, and experiences a sequence of saddle-node bifurcations of limit cycles occurs in the interval $[32.33, 32.4]$. The curve does not extend any further due to computational challenges. 



\item For $n=1$: $\mathscr{C}(E^*,\tau_1^{(1)}, \frac{2\pi}{w(\tau_1^{(1)})})$ and  $\mathscr{C}(E^*,\tau_1^{(2)}, \frac{2\pi}{w(\tau_1^{(2)})})$ coincide, i.e., the periodic solution initials at $\tau_1^{(1)}$ and terminates at $\tau_1^{(2)}$.

\item  For $n=2$: $\mathscr{C}(E^*,\tau_2^{(1)}, \frac{2\pi}{w(\tau_2^{(1)})})$ and $\mathscr{C}(E^*,\tau_2^{(2)}, \frac{2\pi}{w(\tau_2^{(2)})})$ coincide, $\mathscr{C}(E^*,\tau_2^{(3)}, \frac{2\pi}{w(\tau_2^{(3)})})$ and \\ $\mathscr{C}(E^*,\tau_2^{(4)}, \frac{2\pi}{w(\tau_2^{(4)})})$ coincide.
That is to say, we have the formation of a periodic solution at $\tau_{2}^{(1)}$  terminating at $\tau_{2}^{(2)}$, and a formation of another periodic solution at $\tau_{2}^{(3)}$ terminating at  $\tau_{2}^{(4)}$.
\end{itemize}

\end{itemize}

\section{Bifurcation diagram}

In this section, we examine the bifurcation diagram of system \eqref{modelH2} in the $(\tau, K)$-plane. By statement (2) of Lemma \ref{boundary}, a transcritical bifurcation involving $E_K$ and $E^\ast$ occurs when  $\tau=\tau_{max}$, i.e.,  $(\tau, K)$ lies on the curve $\mathscr{L}_c$, defined as
$$\mathscr{L}_c:= \{(\tau, K)\in \overline{\mathbb{R}}_+\times \mathbb{R}_+| K=\frac{ad}{cme^{-d\tau}-d}\} .$$

By Corollary \ref{chopf}, there exists a sequence of bifurcation curves through which Hopf bifurcation occurs at $E^\ast$. Let $\ell_n(\tau, K):=\tau w(\tau, K)-\theta_n(\tau. K)$. The Hopf bifurcation curve is given by
$$\mathscr{L}_n= \{(\tau, K)\in\overline{\mathbb{R}}_+\times \mathbb{R}_+|\ell_n(\tau, K)=0, 0\leq \tau<\bar{\tau}\}, \ n=0, 1, 2, \cdots.$$

\begin{lemma}\label{L0} Consider the curve $\mathscr{L}_0$, i.e., the graph of $\ell_0(\tau, K)=0$.
\begin{enumerate}

\item[(1)] $\mathscr{L}_0$ lies above the curve $\mathscr{L}_c$.

\item[(2)] $\mathscr{L}_0$ starts at the point $(\tau, K)=( 0, K_0)$ and extends to infinity, meaning $\mathscr{L}_0$ exists for all $K\geq K_0$.

\item[(3)] There exists a smooth function $K=K(\tau)$ for $\tau\in[0, \varepsilon)$, where $\varepsilon>0$ is small such that $$K_0=K(0),\ \  \ell_0(K(\tau), \tau)=0, \ \tau\in [0, \varepsilon).$$ Furthermore,  $K'(\tau)<0$ if $cm>(1+\sqrt{2})d$, $K'(\tau)>0$ if $cm<(1+\sqrt{2})d$. If $cm=(1+\sqrt{2})d$, then $K''(\tau)>0$.

\end{enumerate}


\end{lemma}

\begin{proof}
Statement (1) follows directly from Lemma \ref{lemmaE}.

(2) Since $\ell_0(0, K_0)=0-\arccos(1)=0$, the point $(0, K_0)$ serves as an endpoint of $\mathscr{L}_0$. By Lemma \ref{sign of theta}, for any $K\geq K_0$, $\tau w(\tau)$ intersects $\theta_0(\tau)$ at least once. Hence, the graph of $\ell_0(\tau, K)=0$ exists for all $K\geq K_0$.

(3) By direct calculation, we obtain $$\ell_0(0, K_0)=0, \text{and} \ \frac{\partial \ell_0(\tau, K)}{\partial K}_{|(0, K_0)}=-\frac{(cm-d)\sqrt{rd}}{mca\sqrt{c^2m^2-d^2}}\neq 0.$$
By the Implicit Function Theorem, there exists a function $K=K(\tau)$ for $\tau\in[0, \varepsilon)$, where $\varepsilon>0$ is small such that $K_0=K(0)$ and $\ell_0(K(\tau), \tau)=0$. Furthermore, a straightforward calculation yields
 $$K'(\tau)_{|\tau=0}=-\frac{mca(c^2m^2-2cmd-d^2)}{(cm-d)^2}.$$
If $K'(\tau)_{|\tau=0}=0$ (i.e., $cm=(1+\sqrt{2})d$), then by implicit differentiation $K''(\tau)_{|\tau=0}=3\sqrt{2}ac^2m^2$. This establishes the desired result.
\end{proof}

\begin{figure}[!htp]
    \centering
   \subfigure[$cm\leq(1+\sqrt{2})d$]{\includegraphics[angle=0, width=0.40\textwidth]{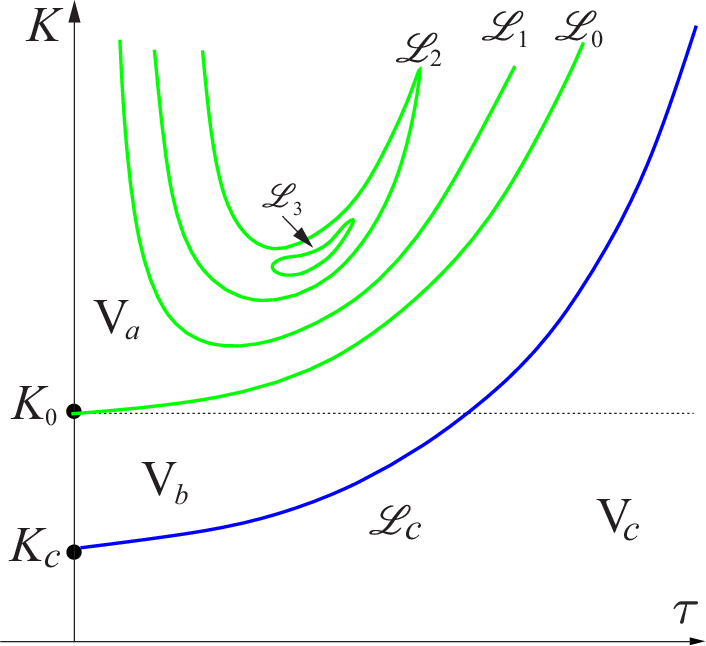}} \qquad
   \subfigure[$cm>(1+\sqrt{2})d$]{\includegraphics[angle=0, width=0.40\textwidth]{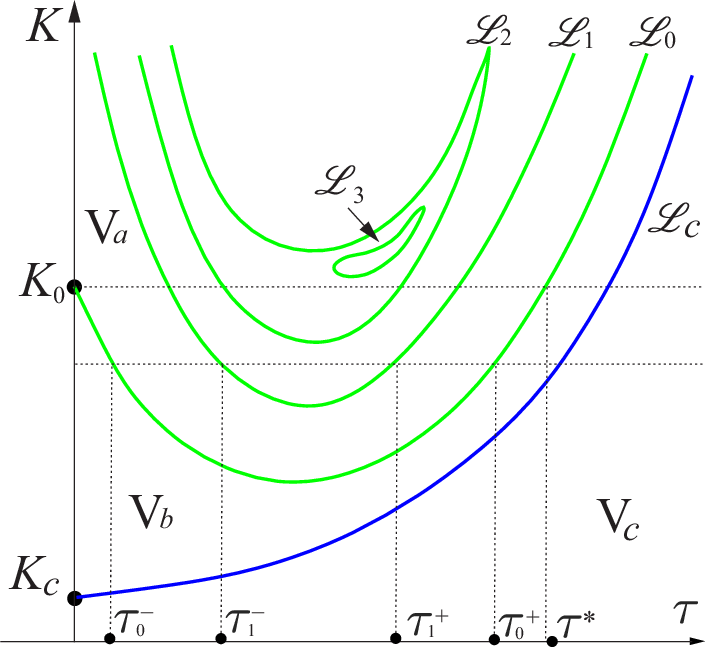}}
\caption{Bifurcation diagram in the $(\tau, K)$-plane. The curve $\mathscr{L}_c$ represents the transcritical bifurcation, while
$\mathscr{L}_n (n=0, 1, 2, 3)$ are curves of Hopf bifurcation. Dash lines highlight critical parameter values that are not bifurcation curves.}\label{bf2}
\end{figure}

In what follows, we will sketch the bifurcation diagram of the system \eqref{modelH2} in the first quadrant of the $(\tau, K)$ plane. There are a curve of transcritical bifurcation $\mathscr{L}_c$, and a curve of Hopf bifurcation $\mathscr{L}_0$. By statement (3) of Lemma \ref{L0}, in a neighborhood of $(0, K_0)$, the curve $\mathscr{L}_0$ exhibits different monotonicity behaviors depending on whether $cm\leq (1+\sqrt{2})d$ or $cm> (1+\sqrt{2})d$. The simplest sketch of $\mathscr{L}_0$ is presented in Fig. \ref{bf2}, ensuring minimal changes in monotonicity for $\tau\in [0, \check{\tau})$.

The two curves $\mathscr{L}_c$ and  $\mathscr{L}_0$ divide the first quadrant of $(\tau, K)$-plane into three regions $V_a$, $V_b$ and $V_c$, defined as follows:
\begin{eqnarray*}
  &&V_a=\{(\tau, K)\in\overline{\mathbb{R}}_+\times \mathbb{R}_+| \ell_0(\tau, K)>, 0\leq \tau<\bar{\tau}\},\\
  &&V_b=\{(\tau, K)\in\overline{\mathbb{R}}_+\times \mathbb{R}_+|\ell_0(\tau, K)<0, 0\leq \tau<\bar{\tau}\},
  V_c=\{(\tau, K)\in \overline{\mathbb{R}}_+\times \mathbb{R}_+|\tau>\tau_{max}\}.
\end{eqnarray*}
By Theorem \ref{no-extended-branches}, $E_K$ is globally asymptotically stable for $(K, \tau)\in V_c$, and system \eqref{modelH2} undergoes a transcritical bifurcation for $(K, \tau)\in \mathscr{L}_c$. $E^\ast$ is locally asymptotically stable  for $(K, \tau)\in V_b$, then system \eqref{modelH2} experiences a Hopf bifurcation for $(K, \tau)\in \mathscr{L}_0$. $E^\ast$ is unstable  for $(K, \tau)\in V_a$, and a periodic solution exists for  $(K, \tau)$ in a tubular neighborhood of $\mathscr{L}_0$.

The existence of additional Hopf bifurcation curves $\mathscr{L}_n (n=1, 2, \cdots)$ depends on the choice of parameters. If such curves exist, they are positioned above $\mathscr{L}_0$ and do not intersect any other Hopf bifurcation curve.  Numerical studies suggest two possible topological structures for $\mathscr{L}_n$: $\mathscr{L}_n\simeq \mathbb{R}$, or $\mathscr{L}_n\simeq \mathbb{S}^1$, where $\simeq$ means ``is topologically isomorphic''. Fig. \ref{bf2} illustrates an example where $\mathscr{L}_1\simeq \mathbb{R}$, $\mathscr{L}_2\simeq \mathbb{R}$, $\mathscr{L}_3\simeq \mathbb{S}^1$.

The region $V_a$ is divided into several subregions by the Hopf bifurcation curves $\mathscr{L}_n (n=1, 2, \cdots)$. Whenever   $(K, \tau)$ crosses $\mathscr{L}_n$, the Hopf bifurcation occurs, and a periodic solution emerges in a tubular neighborhood of $\mathscr{L}_n$. For $K<K_0$, if $\mathscr{L}_n (n=2, 3, \cdots)$ transversally intersects the horizontal line exactly twice, by Theorem \ref{smallK}, the connected components $\mathscr{C}(E^*,\tau_n^{-}, \frac{2\pi}{w(\tau_n^{-})})$ and  $\mathscr{C}(E^*,\tau_n^{+}, \frac{2\pi}{w(\tau_n^{+})})$ coincide. That is,  a periodic solution emerging at $(\tau_n^-, K)$ will terminate at $(\tau_n^+, K)$. Determining the number of periodic solutions within the subregions between successive Hopf bifurcation curves is challenging. This is due to the potential presence of other bifurcations affecting periodic solutions. 
Even in the absence of such bifurcations, tracing the periodic solutions as $(\tau,  K)$ crosses $\mathscr{L}_n (n=0, 1, 2 \cdots)$ remains difficult, as there is no definitive criterion for predicting the onset and termination of connected components in global Hopf bifurcation when $K>K_0$ or $\chi(n)>2$.

\begin{remark}
   The bifurcation diagram in Fig. \ref{bf2} is not complete, as there may exist additional bifurcation curves of periodic solutions, such as saddle-node bifurcations of limit cycles and period-doubling bifurcations. A detailed investigation of these bifurcations is beyond the scope of this work.
\end{remark}


\begin{remark}\label{R5}
If $cm>(1+\sqrt{2})d$, define
 $$\tau^\ast=\arg \min_{\tau\in(0, \bar{\tau})}\{\ell_0(\tau, K_0)=0\}.$$
Then, for any point $(\tau, K)\in\mathscr{L}_0$ with $0<\tau<\tau^\ast$, we have $K<K_0$ (see Fig. \ref{bf2} (b)). 
\end{remark}

\noindent {\bf Biological significance of $\tau^\ast$.} {\it Predator-prey coexistence via oscillatory patterns can occur at lower environmental carrying capacities when delay ($\tau > 0$) is introduced, compared to scenarios without delay ($\tau = 0$) under the condition that $(1+\sqrt{2})d<cm$, and $0<\tau<\tau^\ast$.} The significance of $\tau^\ast$ will be further illustrated in the next section.

\section{Discussion}

The classical predator-prey model with Holling type II functional response typically exhibits relatively simple dynamical behavior. By contrast, the incorporation of a maturation delay introduces significantly richer dynamics and provides a more realistic framework for modeling predator-prey interactions observed in nature. In this section, we interpret the effects of delay on population persistence and oscillatory coexistence.

\setlength{\leftmargini}{0.3cm}
\begin{itemize}
    \item {\it Threshold dynamics for prey--predator coexistence.}

The coexistence of prey and predators requires $\tau<\tau_{\max}$, which is equivalent to $K>K_1$. Here
\[
    K_1=\frac{ad}{cme^{-d\tau}-d},
\]
and thus
\[
    K_1>\frac{ad}{cme^{-d\tau}-d}\Big|_{\tau=0}=K_c.
\]
As $\tau$ increases, $K_1$ increases, meaning that a higher environmental carrying capacity is needed for both species to coexist. This threshold is illustrated by the transcritical bifurcation curve $\mathscr{L}_c$ in Fig.~\ref{bf2}, together with the critical values $K_1$ and $K_c$ shown in Fig.~\ref{kk}. Biologically, this behavior is intuitive: larger delays reduce predator survival ($e^{-d\tau}$), forcing predators to rely on greater prey abundance to compensate for mortality. As a consequence, prey populations face stronger predation pressure.

 \item {\it Oscillatory dynamics of prey and predator populations.}

Oscillatory dynamics emerge once the condition $\ell_0(\tau, K)>0$ is satisfied. A secondary threshold involving the predator death rate further determines whether sustained oscillations are possible.

\begin{enumerate}
\item[(1)] If $cm\leq (1+\sqrt{2})d$, predator mortality is relatively high, leading to low survival $e^{-d\tau}$. In this case, predators require more prey to maintain oscillations. Since survival decreases as $\tau$ increases, the critical carrying capacity rises monotonically with $\tau$. This behavior is reflected by the monotone increase of $\mathscr{L}_0$ in Fig.~\ref{bf2}(a).

\item[(2)] If $cm>(1+\sqrt{2})d$, predator mortality is comparatively low, yielding higher survival rates. For small delays $\tau \in(0,\tau^\ast)$, oscillatory coexistence can be sustained at lower values of $K$. However, once $\tau$ exceeds $\tau^\ast$ (with $\tau \in(\tau^\ast,\bar{\tau})$), survival declines sufficiently to require larger $K$ values for oscillations to persist. This non-monotone behavior of $\mathscr{L}_0$ is illustrated in Fig.~\ref{bf2}(b), with $\tau^\ast$ defined in Remark~\ref{R5}.
\end{enumerate}
\end{itemize}

\begin{figure}[!ht]
\begin{center}
\includegraphics[angle=0, width=0.76\textwidth]{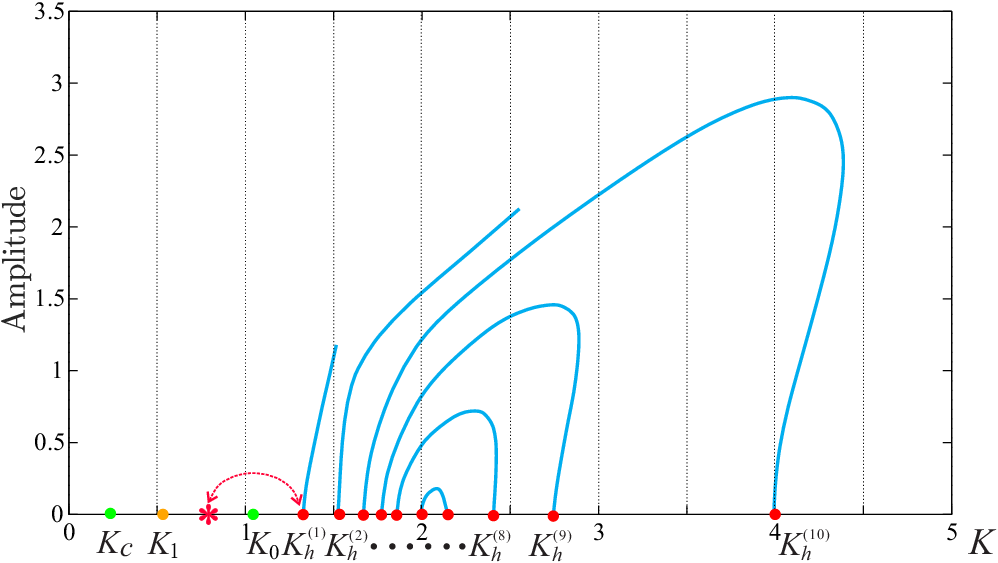}
\caption{\small{Global Hopf bifurcation and connected components as functions of $K$ for $\tau>\tau^\ast$. The orange dot corresponds to the transcritical bifurcation threshold $K_1$, while the red dots denote the Hopf bifurcation thresholds $K_h^{(i)}$ ($i=1,2,\dots,10$). The location of $K_h^{(1)}$ for $\tau \in (0,\tau^\ast)$ is indicated by ``${\color{red}{\ast}}$''. Critical values $K_c$ and $K_0$ denote the transcritical and Hopf bifurcation thresholds of the ODE model ($\tau=0$). Detailed descriptions are provided in this section.}}\label{kk}
\end{center}
\end{figure}

To illustrate these effects, consider the bifurcation diagram of Fig.~\ref{kk}, where $K$ is treated as the bifurcation parameter. Fixing $\tau=26$, $r=30$, $m=1$, $c=4$, $d=0.1$, and $a=1$, we have $cm>(1+\sqrt{2})d$ and $\tau>\tau^\ast=24.1$. The critical bifurcation values are
\[
K_c=0.25,\quad K_1=0.51,\quad K_0=1.0512,\quad K_h^{(1)}=1.328,\quad K_h^{(2)}=1.54,\quad K_h^{(3)}=1.67,\quad K_h^{(4)}=1.76,
\]
\[
K_h^{(5)}=1.85,\quad K_h^{(6)}=1.994,\quad K_h^{(7)}=2.135,\quad K_h^{(8)}=2.3956,\quad K_h^{(9)}=2.7395,\quad K_h^{(10)}=3.996.
\]
As $K$ increases, system \eqref{modelH2} undergoes a transcritical bifurcation at $K=K_1$ and a sequence of Hopf bifurcations at $K=K_h^{(i)}$. In particular, $K_h^{(1)}>K_0$ and $K_1>K_c$, showing that delay increases the environmental resources needed for both coexistence and oscillatory dynamics.

When $\tau$ decreases into the interval $\tau\in(0,\tau^\ast)$, the first Hopf bifurcation threshold $K_h^{(1)}$ shifts to the position marked ``${\color{red}{\ast}}$'' between $K_1$ and $K_0$ in Fig.~\ref{kk}. This indicates that smaller delays may allow oscillatory coexistence with fewer resources.

In summary, predator populations are highly sensitive to delays arising from environmental or climatic factors such as extended winters, elevated temperatures, or resource scarcity. Such delays may either increase the resource requirements for persistence or, paradoxically, facilitate coexistence under more restrictive conditions. From an ecological perspective, these results underscore that small perturbations to maturation timing can drive ecosystems toward coexistence or collapse, with potential implications for both conservation and pest management strategies.

\section*{Conflicts of Interest}

The authors declare that they have no conflicts of interest.

\section*{Acknowledgments}
The research of C. Shan was partially supported by Simons Foundation-Collaboration Grants for Mathematicians 523360. The research of H. Wang was partially supported by the Natural Sciences and Engineering Research Council of Canada (Individual Discovery Grant RGPIN-2025-05734 and Discovery Accelerator Supplement Award RGPAS-2020-00090) and the Canada Research Chairs program (Tier 1 Canada Research Chair Award). We used ChatGPT to polish some of the writing.

\end{document}